\documentclass[11pt, a4]{amsart}

\usepackage{amsmath}
\usepackage{amsthm}
\usepackage{amssymb}
\usepackage{amscd}
\usepackage{subfigure}
\usepackage{epsfig}
\usepackage{graphicx}

\def\N{\mathbb N}
\def\Z{\mathbb Z}

\def\R{\mathbb R}

\def\C{\mathbb C}

\newcommand{\be}{\begin{eqnarray}}
\newcommand{\ee}{\end{eqnarray}}
\newcommand{\Tk}{{\mathcal T}}
\newcommand{\Gk}{{\mathcal G}}
\newcommand{\Sk}{{\mathcal S}}

\newcommand{\Jk}{{\mathcal J}}
\newcommand{\supp}{\mbox{\rm supp}}
\newcommand{\dens}{\mbox{\rm dens}}
\newcommand{\Vol}{\mbox{\rm Vol}}
\newcommand{\Lb}{\mbox {\boldmath ${\Lambda}$}}
\newcommand{\Lbs}{\mbox{\scriptsize\boldmath ${\Lambda}$}}

\newcommand{\mmax}{\rm max}

\newcommand{\diam}{\mbox{\rm diam}}
\newcommand{\Lam}{{\Lambda}}
\newcommand{\Pb}{\mbox {\bf P}}
\newcommand{\Gb}{\mbox {\boldmath ${\Gamma}$}}

\newcommand{\Dk}{{\mathcal D}}
\newcommand{\Ak}{{\mathcal A}}

\newtheorem{theorem}{Theorem}[section]
\newtheorem{thm}[theorem]{Theorem}
\newtheorem{prop}[theorem]{Proposition}
\newtheorem{lem}[theorem]{Lemma}

\newtheorem{conj}[theorem]{Conjecture}

\newtheorem{defi}[theorem]{Definition}
\newtheorem{rem}[theorem]{Remark}
\newtheorem{ex}[theorem]{Example}

\numberwithin{equation}{section}

\begin{document}
\title[Pure pointedness of self-affine tilings]{Algorithm for
determining pure pointedness of self-affine tilings }

\date{\today}

\thanks{The first author is supported by the Japanese Society for the Promotion of Science (JSPS), grant in aid 21540010. The second author is grateful for the support of
KIAS in this research.}

\maketitle



 \centerline{
{ Shigeki Akiyama $^{\,\rm a}$ and Jeong-Yup Lee $^{\,\rm b}$}}

\hspace*{4em}

\smallskip

{\footnotesize
\hspace*{4em} a: Department of Mathematics, Faculty of Science, Niigata University, \\
\hspace*{4em} \hspace*{2.5em} 8050 Ikarashi-2, Nishi-ku Niigata,
Japan (zip: 950-2181) \\
\hspace*{4em} \hspace*{2.5em}
\email{akiyama@math.sc.niigata-u.ac.jp}

\smallskip
\hspace*{4em} b: KIAS 207-43, Cheongnyangni 2-dong, Dongdaemun-gu,
\\ \hspace*{4em} \hspace*{2.5em} Seoul 130-722, Korea \\
\hspace*{4em} \hspace*{2.5em} \email{jylee@kias.re.kr} }

\begin{abstract}
Overlap coincidence in a self-affine tiling in $\R^d$ is
equivalent to pure point dynamical spectrum of the tiling
dynamical system. We interpret the overlap coincidence in the
setting of substitution Delone set in $\R^d$ and find an efficient
algorithm to check the pure point dynamical spectrum. This
algorithm is easy to implement into a computer program. We give
the program and apply it to several examples. In the course the
proof of the algorithm, we show a variant of the conjecture of
Urba\'nski (Solomyak \cite{Solomyak:08}) on the Hausdorff
dimension of the boundaries of fractal tiles.
\end{abstract}

\hspace*{4em}

{\footnotesize Keywords: Pure point spectrum, Self-affine tilings,
Coincidence, Substitution Delone sets, Meyer sets, \\
\hspace*{4em} \hspace*{1.5em} Algorithm, Quasicrystals, Hausdorff
dimension, Fractals.}

\section{Introduction}

To model self-inducing structures of dynamical systems, symbolic
dynamical systems associated with substitutions play an important
role and many works describe their spectral properties and
geometric realizations (see \cite{Fogg}). To extend the symbolic
substitutive systems to higher dimensions, self-affine tiling
dynamical systems are studied in detail in \cite{soltil, LMS2} and
many related studies are done along this line. These tiling
dynamical systems share many properties with the symbolic
substitutive systems and are intimately related to the explicit
construction of Markov partitions. It is a subtle question to
determine whether a given tiling dynamical system has pure point
dynamical spectrum or not. It is known from \cite{soltil, LMS2}
that `overlap coincidence' (see
Def.\,\ref{def-overlapCoincidence}) is an equivalent criterion to
check this. However the overlap coincidence was not easy to
compute there in practice because it requires topological
properties of the tiles.
To settle this difficulty, we shall employ the duality between
self-affine tilings and substitutive Delone sets
\cite{lawa,LMS2,Lee}. An aim of this paper is to transfer the
overlap coincidence to substitution Delone sets, find a computable
algorithm to check the pure pointedness and implement it into a
program language.

Further motivation to show the pure pointedness comes from the
study of aperiodic order. It is an interesting question to ask
what kind of point sets, modeling atomic configurations, present
pure point diffraction. This is related with the understanding of
the fundamental structures of quasicrystals. It has been known
from \cite{LMS1, Gouere, BL} that pure point diffraction spectrum
is equivalent to pure point dynamical spectrum in quite a general
setting. So the algorithm we give here can be used for checking
pure point diffraction of general
self-affine quasi-periodic structures.


There are many equivalent criteria to the pure point dynamical
spectrum in literature. Among them, coincidences are very well
known as a characterization of the pure point dynamical spectrum.
There are many different notions of coincidences but basically
they imply the same thing. In $1$-dim substitution sequences,
Dekking's coincidence is well-known for the case of
constant-length substitutions \cite{Dek}. For $1$-dim irreducible
Pisot substitution sequences or tilings, super coincidence, strong
coincidence, geometric coincidence, balanced pairs, and boundary
graph are known \cite{IR,AI,BK,SS,ST,Fogg}. In higher dimensions,
modular coincidence was introduced for lattice substitution Delone
sets \cite{LM1, LMS2,DS}, and overlap coincidence and algebraic
coincidence are known for substitution tilings and substitution
Delone sets under the assumption of Meyer property \cite{soltil,
Lee}. We are going to use the overlap coincidence for computation
here.

We note that it is essential to assume the Meyer property of the
corresponding substitution Delone set. Otherwise, the algorithm
will either not terminate as the number of overlaps becomes
infinite, or terminate with incorrect outputs. It is shown in
\cite{LS} that substitution Delone sets with pure point dynamical
spectrum necessarily have the Meyer property. It is also studied
in \cite{LS2} under which conditions on the expansion maps of the
substitutions, the point sets are guaranteed to have the Meyer
property.

There are a few results in literature for the actual computation
of coincidence. For $1$-dimension unit Pisot substitutions and
self-affine tilings coming from their geometric realizations,
computable algorithm is discussed in \cite{Siegel,ST} using the
boundary graph. For irreducible $1$-dimension Pisot substitution,
balanced pair algorithm is implemented in \cite{SS}.
For higher dimensions, Dekking's coincidence and modular
coincidence are used for the case of lattice substitution Delone
sets in $\R^d$ \cite{Dek, LM1, LMS2}. It was shown in \cite{DS}
that the modular coincidence in lattice substitution Delone sets
can be determined within some bounded iterations. The given bound
is exponential in the number of colours of the Delone sets. It was
conjectured in \cite{DS} that this bound is just polynomial in the
number of colours. We give an affirmative answer to this
conjecture not only for the lattice substitution Delone sets but
also the substitution Delone sets with the Meyer property (see
Remark\,1).

In this paper we compute overlap coincidence for general
self-affine tilings. Our method covers, non-unit cases, higher
dimensional and non-lattice based self-affine tilings.
With regard to computation, already in the original paper by
Solomyak \cite{soltil}, the number of overlaps becomes too large
to handle by hand. Apart from 1-dimensional case with connected
tiles (i.e. intervals), it is quite hard to check whether
translated tiles have intersection. The implementation is already
difficult for polygonal tilings, and moreover, tiles often have
fractal boundaries in higher dimensional cases.
To overcome this difficulty, we escape from judging interior
intersection. We interpret the overlaps in terms of points and
translation vectors, and only care distances between the
corresponding translated tiles. If the distances are within a
rough bound (see (\ref{potential})), we say they are {\it
potential overlaps}. Of course by this change, some pairs of
translated tiles may not intersect, or only meet at their
boundaries. To distinguish these cases from overlaps with
interior intersection, we introduce a {\em potential overlap
graph with multiplicities}. At the expense of having a larger
graph, all computation becomes simple and easy to implement into
computer programs. Showing that our criterion (Theorem \ref{main}
(ii)) is necessary, we prove partially a variant of the conjecture
which asserts that the boundaries of the self-affine tiles have
Hausdorff dimension less than the space dimension $d$ (see
\cite{Solomyak:08} for the conjecture).

The paper is organized in the following way: In Section
\ref{preliminary}, we give definitions and notations. As a main
result, we present a mathematical algorithm computing the overlap
coincidence. In Section \ref{Computing coincidence} and \ref{The
potential overlaps that are not real overlaps}, we give a
justification on this algorithm. In Section \ref{Examples}, we
have built a `Mathematica' program implementing the algorithm and
apply it to $1$, $2$ and $3$-dimensional examples. The spectral
properties of some of the examples have not been known before.

\section{Preliminary} \label{preliminary}

\noindent The notation and terminology we use in this paper is
standard. We refer the reader to \cite{LMS2} for more detailed
definitions and to \cite{LP} for the standard notions.

\subsection{Tilings}

We begin with a set of types (or colours) $\{1,\ldots,m\}$, which
we fix once and for all. A {\em tile} in $\R^d$ is defined as a
pair $T=(A,i)$ where $A=\supp(T)$ (the support of $T$) is a
compact set in $\R^d$, which is the closure of its interior, and
$i=l(T)\in \{1,\ldots,m\}$ is the type of $T$. We let $g+T =
(g+A,i)$ for $g\in \R^d$. We say that a set $P$ of tiles is a {\em
patch} if the number of tiles in $P$ is finite and the tiles of
$P$ have mutually disjoint interiors. The {\em support of a patch}
is the union of the supports of the tiles that are in it. The {\em
translate of a patch} $P$ by $g\in \R^d$ is $g+P := \{g+T:\ T\in
P\}$. We say that two patches $P_1$ and $P_2$ are {\em
translationally equivalent} if $P_2 = g+P_1$ for some $g\in \R^d$.
A {\em tiling} of $\R^d$ is a set $\Tk$ of tiles such that $\R^d =
\bigcup \{\supp(T) : T \in \Tk\}$ and distinct tiles have disjoint
interiors. We always assume that any two $\Tk$-tiles with the same
colour are translationally equivalent (hence there are finitely
many $\Tk$-tiles up to translations). Let $\Xi(\Tk) : = \{x \in
\R^d : \exists \ T, T' \in \Tk, T' = x + T \}$. We say that $\Tk$
has {\em finite local complexity (FLC)} if for each radius $R > 0$
there are only finitely many equivalent classes of patches whose
support lies in some ball of radius $R$. We define $\Tk \cap A :=
\{T \in \Tk : \supp(T) \cap A \neq \emptyset\}$ for a bounded set
$A \subset \R^d$. We say that $\Tk$ is {\em repetitive} if for
every compact set $K \subset \R^d$, $\{t \in \R^d : \Tk \cap K =
(t + \Tk) \cap K\}$ is relatively dense. We write $B_R(y)$ for the
closed ball of radius $R$ centered at $y$ and use also $B_R$ for
$B_R(0)$.

\subsection{Delone multi-colour sets}

\noindent A {\em multi-colour set} or {\em $m$-multi-colour set}
in $\R^d$ is a subset $\Lb = \Lam_1 \times \dots \times \Lam_m
\subset \R^d \times \dots \times \R^d$ \; ($m$ copies) where
$\Lam_i \subset \R^d$. We also write $\Lb = (\Lam_1, \dots,
\Lam_m) = (\Lam_i)_{i\le m}$. Recall that a Delone set is a
relatively dense and uniformly discrete subset of $\R^d$. We say
that $\Lb=(\Lambda_i)_{i\le m}$ is a {\em Delone multi-colour set}
in $\R^d$ if each $\Lambda_i$ is Delone and
$\supp(\Lb):=\cup_{i=1}^m \Lambda_i \subset \R^d$ is Delone. A
{\em cluster} of $\Lb$ is, by definition, a family $\Pb =
(P_i)_{i\le m}$ where $P_i \subset \Lambda_i$ is finite for all
$i\le m$. The translate of a cluster $\Pb$ by $x \in \R^d$ is $x +
\Pb = (x + P_i)_{i \le m}$. We say that two clusters $\Pb$ and
$\Pb'$ are translationally equivalent if $\Pb = x + \Pb'$ for some
$x \in \R^d$. We say that $\Lam \subset \R^d$ is a {\em Meyer set}
if it is a Delone set and $\Lam - \Lam$ is uniformly discrete (\cite{Lag}).
We define FLC and repetitivity on Delone multi-colour sets in the
same way as the corresponding properties on tilings. The types (or
colours) of points on Delone multi-colour sets have the same
concept as the colours of tiles on tilings.

\subsection{Substitutions} \label{substi-pointSet-tilings}

We say that a linear map $Q : \R^d \rightarrow \R^d$ is {\em
expansive} if all the eigenvalues of $Q$
 lie outside the closed unit disk in $\C$.
\subsubsection{Substitutions on tilings}

\begin{defi}\label{def-subst}
{\em Let $\Ak = \{T_1,\ldots,T_m\}$ be a finite set of tiles in $\R^d$
such that $T_i=(A_i,i)$; we will call them {\em prototiles}.
Denote by ${\mathcal{P}}_{\Ak}$ the set of non empty
patches. 
We say that $\Omega: \Ak \to {\mathcal{P}}_{\Ak}$ is a {\em tile-substitution} (or simply
{\em substitution}) with
an expansive map $Q$ if there exist finite sets $\Dk_{ij}\subset \R^d$ for
$i,j \le m$ such that
\begin{equation}
\Omega(T_j)=
\{u+T_i:\ u\in \Dk_{ij},\ i=1,\ldots,m\}
\label{subdiv}
\end{equation}
with
\begin{eqnarray} \label{tile-subdiv}
Q A_j = \bigcup_{i=1}^m (\Dk_{ij}+A_i) \ \ \  \mbox{for} \  j\le m.
\end{eqnarray}
Here all sets in the right-hand side must have disjoint interiors;
it is possible for some of the $\Dk_{ij}$ to be empty.}
\end{defi}

\noindent
Note that $Q A_j = \supp (\Omega(T_j)) = Q \supp(T_j)$.
The substitution (\ref{subdiv}) is extended to all translates of prototiles by
\be \label{tile-equivalence}
\Omega(x+T_j)= Q x + \Omega(T_j),
\ee
in particular,
\be \label{supp-of-iterated-tile}
\supp (\Omega(x + T_j)) & = & \supp(Qx + \Omega(T_j)) \nonumber \\
                        & = & Qx + Q \supp(T_j) \nonumber \\
                        & = & Q(x + \supp(T_j)),
\ee
and to patches and tilings by
$\Omega(P)=\cup\{\Omega(T):\ T\in P\}$.
The substitution $\Omega$ can be iterated, producing larger and larger patches
$\Omega^k(P)$.
We say that $\Tk$ is a {\em substitution tiling} if $\Tk$ is a
tiling and $\Omega(\Tk) = \Tk$ with some substitution $\Omega$. In
this case, we also say that $\Tk$ is a {\em fixed point} of
$\Omega$.
We say that substitution Delone multi-colour set is {\em
primitive} if the corresponding substitution matrix $S$, with
$S_{ij}= \sharp (\Dk_{ij})$, is primitive. A repetitive fixed
point of a primitive tile-substitution with FLC is called a {\em
self-affine tiling}. If $Q$ is a similarity, then the tiling will
be called {\em self-similar}. For any self affine tiling which
holds (\ref{tile-subdiv}), we define $\Phi$ an $m \times m$ array
for which each entry is $\Phi_{ij}$,
\[ \Phi_{ij} = \{ f : x \mapsto Qx + d \, : \,d \in \Dk_{ij}\}\,\]
and call $\Phi$ a {\em matrix function system (MFS)} for the
substitution $\Omega$.

\subsubsection{Substitutions on Delone multi-colour sets}

\begin{defi} \label{def-subst-mul}
{\em $\Lb = (\Lam_i)_{i\le m}$ is called a {\em substitution
Delone multi-colour set} in $\R^d$ if $\Lb$ is a Delone
multi-colour set and there exist an expansive map $Q:\, \R^d\to
\R^d$ and finite sets $\Dk_{ij}$ for $i,j\le m$ such that \be
\label{eq-sub} \Lambda_i = \bigcup_{j=1}^m (Q \Lambda_j +
\Dk_{ij}),\ \ \ i \le m, \ee where the unions on the right-hand
side are disjoint.}
\end{defi}

We say that a cluster $\Pb$ is {\em legal} if it is a translate of
a subcluster of a cluster generated from one point of $\Lb$, i.e.
$a + \Pb \subset \Phi^k (x)$ for some $k \in \Z_+$, $a \in \R^d$
and $x \in \Lb$.

\subsubsection{Representability of $\Lb$ as a tiling} \label{RepresentabilityAsTiling}

Let $\Lb$ be a primitive substitution Delone multi-colour set. One
can set up an {\em adjoint system of equations} \be \label{eq-til}
Q A_j = \bigcup_{i=1}^m (\Dk_{ij} + A_i),\ \ \ j \le m \ee from
the equation (\ref{eq-sub}). It is known that (\ref{eq-til})
always has a unique solution for which $\{A_1, \dots, A_m\}$ is a
family of non-empty compact sets of $\R^d$. It is proved in
\cite[Th.\,2.4 and Th.\,5.5]{lawa} that if $\Lb$ is a primitive
substitution Delone multi-colour set, all the sets $A_i$ from
(\ref{eq-til}) have non-empty interiors and, moreover, each $A_i$
is the closure of its interior. We say that $\Lb$ is {\em
representable} (by tiles) if
\[\Lb + \Ak := \{x + T_i :\ x\in
\Lambda_i,\ i \le m\}\] is a tiling of $\R^d$, where $T_i =
(A_i,i)$, $i \le m$, for which $A_i$'s arise from the solution to
the adjoint system (\ref{eq-til}) and $\Ak = \{T_i : i\le m\} $.
Then $\Lb + \Ak$ is a substitution tiling and we can define a
tile-substitution $\Omega$ satisfying
\[\Omega(\Lb + \Ak) = \Lb + \Ak \] from (\ref{eq-til}). We call $\Lb + \Ak$ the
{\em associated
substitution tiling} of $\Lb$. Let $\Phi = (\Phi_{ij})$ be a MFS
for $\Omega$. For any subset $\Gb = (\Gamma_j)_{j \le m} \subset
\Lb$, $\Phi_{ij}(\Gamma_j) = Q \Gamma_j + \Dk_{ij}$, for $j \le
m$. Let $\Phi(\Gb) = (\cup_{j \le m} \Phi_{ij}(\Gamma_j))_{i \le
m}$. Then $\Phi_{ij}(\Lam_j) = Q \Lam_j + \Dk_{ij}$, where $i \le
m$. For any $k \in \Z_+$ and $x \in \Lam_j$ with $j \le m$, we let
$\Phi^k (x) = \Phi^{k-1}((\Phi_{ij}(x))_{i \le m})$. Note that for
any $k \in \Z_+$, $\Phi^k (\Lam_j) = (Q^k \Lam_j +
(\Dk^k)_{ij})_{i \le m}$ where
\[(\Dk^k)_{ij} = \bigcup_{n_1,n_2,\dots,n_{(k-1)} \le m}
(\Dk_{in_1} + Q \Dk_{n_1 n_2} + \cdots + Q^{k-1} \Dk_{n_{(k-1)} j})
\]
and $\Phi^k (\Lb) = \Lb$.

In \cite[Lemma 3.2]{lawa} it is shown that if $\Lb$ is a
substitution Delone multi-colour set, then there is a finite
multi-colour set (cluster) $\Pb \subset \Lb$ for which
$\Phi^{n-1}(\Pb) \subset \Phi^n(\Pb)$ for $n \ge 1$ and $\Lb =
\lim_{n \to \infty} \Phi^n (\Pb)$. We call such a multi-colour set
$\Pb$ a {\em generating set} for $\Lb$.

\begin{thm}{\em \cite{LMS2}}
Let $\Lb$ be a repetitive primitive substitution Delone multi-colour set in $\R^d$. Then every
$\Lb$-cluster is legal if and only if $\Lb$ is representable.
\end{thm}

On the other hand, if a self-affine tiling $\Tk = \{T_j+ \Lam_j :
j \le m \}$ is given, we get an associated substitution Delone
multi-colour set $\Lb_{\Tk} = (\Lam_i)_{i \le m}$ of $\Tk$ (see
\cite[Lemma\,5.4]{Lee}).


\subsection{Pure point spectrum and overlap coincidence}

Let $\Tk$ be a self-affine tiling in $\R^d$. We define the space
of tilings as the orbit closure of $\Tk$ under the translation
action: $X_{\Tk} = \overline{\{-h + \Tk : h \in \R^d \}}$, in the
well-known ``local topology'': for a small $\epsilon
> 0$ two point sets $\Sk_1, \Sk_2$ are $\epsilon$-close if $\Sk_1$
and $\Sk_2$ agree on the ball of radius $\epsilon^{-1}$ around the
origin, after a translation of size less than $\epsilon$. The
group $\R^d$ acts on $X_{\Tk}$ by translations which are obviously
homeomorphisms, and we get a topological dynamical system
$(X_{\Tk},\R^d)$. Let $\mu$ be an ergodic invariant Borel
probability measure for the dynamical system $(X_{\Tk},\R^d)$. We
consider the associated group of unitary operators $\{U_g\}_{g\in
\R^d}$ on $L^2(X_{\Tk},\mu):$
\[U_g f(\Sk) = f(-g+\Sk).\]
A vector $\alpha =(\alpha_1,\ldots,\alpha_d) \in \R^d$ is said to
be an eigenvalue for the $\R^d$-action if there exists an
eigenfunction $f\in L^2(X_{\Tk},\mu),$ that is, $\ f\not\equiv 0$
and
\[U_g f = e^{2 \pi i g \cdot \alpha} f,\ \ \ \mbox{for all}
\ \ g\in \R^d.\] The dynamical system $(X_{\Tk},\mu,\R^d)$ is said
to have {\em pure point(or pure discrete) spectrum} if the linear
span of the eigenfunctions is dense in $L^2(X_{\Tk}, \mu)$. Recall
that a topological dynamical system of a self-affine tiling is
{\em uniquely ergodic} i.e. there is a unique invariant
probability measure \cite{LMS2}.

\subsection{Overlaps}

Overlap and overlap coincidence are originally defined with tiles
in substitution tilings \cite{soltil}. For computational reason,
we define overlaps with the corresponding representative points of
tiles here. A triple $(u, y, v)$, with $u + T_i, v + T_j \in \Tk$
and $y \in \Xi(\Tk)$, is called an {\em overlap} (or {\em real
overlap}) if
\[ (u+A_i-y)^{\circ} \cap (v+A_j)^{\circ} \neq \emptyset, \]
where $A_i = \supp(T_i)$ and $A_j = \supp(T_j)$. We define
$(u+A_i-y) \cap (v+A_j)$ the {\em support of an overlap} $(u, y,
v)$ and denote it by $\supp(u,y,v)$. We say that two overlaps $(u,
y, v)$ and $(u', y', v')$ are {\em equivalent} if there exists $g
\in \R^d$ such that $u-y = g+u'-y'$ and $v= g+ v'$, where
$u+T_i,u'+T_i \in \Tk$ and $v+T_j, v'+T_j \in \Tk$ for some $1 \le
i,j \le m $. Denote by $[(u, y, v)]$ the equivalence class of an
overlap. An overlap $(u, y, v)$ is a {\em coincidence} if
\[ \mbox{$u-y = v$ and $u + T_i, v + T_i \in \Tk$ for some $i \le m$}.\]
Let $\mathcal{O} = (u, y, v)$ be an overlap in $\Tk$, we define
{\em $k$-th inflated overlap}
\begin{eqnarray*}
 {\Phi}^k \mathcal{O} = \{(u', Q^k y, v') \, :
  u' \in \Phi^k(u), v' \in \Phi^k(v), \ \mbox{and $(u',Q^ky,v')$ is an overlap} \}.
\end{eqnarray*}

\begin{defi} \label{def-overlapCoincidence}
{\em We say that a self-affine tiling $\Tk$ admits an {\em overlap
coincidence} if there exists $\ell \in \Z_+$ such that for each
overlap $\mathcal{O}$ in $\Tk$, ${\Phi}^{\ell} \mathcal{O}$
contains a coincidence.}
\end{defi}

\begin{theorem} \cite{LMS2, Lee}
Let $\Tk$ be a self-affine tiling in $\R^d$ such that $\Xi(\Tk)$
is a Meyer set. Then $(X_{\Tk}, \R^d, \mu)$ has a pure point
dynamical spectrum if and only if $\Tk$ admits an overlap
coincidence.
\end{theorem}

In actual computation, it is not easy to determine whether a given
triple is an overlap, since two points can be very close without
having the interiors of the corresponding tiles meet. So we
introduce a notion of {\it potential overlaps}.

Let $\xi \in \R^d$ be a fixed point under the substitution such
that $\xi \in \Phi(\xi)$. When there is no confusion, we will
identify $\xi$ with a coloured point $(\xi,i)$ in $\Lb_{\Tk}$ with
$\xi \in \R^d$. We find a basis of $\R^d$ \[\mathcal{B} = \{
\alpha_1, \dots, \alpha_d \} \subset \Xi(\Tk)\] such that \be
\label{translates-basis} \xi + \alpha_1, \dots, \xi + \alpha_d \in
\Phi^{\ell}(\xi) \ \ \mbox{for some $\ell \in \Z_+$}. \ee

Let
$$\alpha_{max}:= {\mmax}\{ |\alpha_i| :
\alpha_i \in \mathcal{B} \}. $$ For any $n \in \Z_+$, let
\[ e^{(n)}:= {\mmax} \{|d_{ij} - {d'}_{k\ell}|: d_{ij} \in (\mathcal{D}^n)_{ij}, {d'}_{k\ell} \in (\mathcal{D}^n)_{k\ell},\ \mbox{where} \ 1 \le
i,j,k,\ell \le m\},\]
Let $\Vert \cdot \Vert$ be the operator norm induced by Euclidean norm.
Since $Q$ is an expansive map, we can find $k \in \Z_+$ such that
$$ \Vert Q^{-k}\Vert < 1.$$
Note that for any $v \in \R^d$, \be
\label{find-K-corresponding-to-h}|Q^k v| \ge \frac{1}{\Vert Q^{-k}\Vert}
|v|. \ee Let \be \label{formula-diam} R = \frac{ e^{(k)} \cdot
\Vert Q^{-k}\Vert}{1 - \Vert Q^{-k}\Vert}. \ee
We say that a triple $(u,y,v)$, with $u + T_i, v + T_j \in \Tk$
for some $i,j \le m$ and $y \in \Xi(\Tk)$, is called a {\em
potential overlap} if
\begin{equation}
\label{potential}
|u-y-v| \le R
\end{equation}
and we say that the potential overlap $(u,y,v)$ occurs by the
translation $y$.

\begin{lem}
If $(u, y, v)$ is an overlap, then $(u, y, v)$ is a potential
overlap.
\end{lem}

\begin{proof}
From (\ref{eq-til}), we get
\[ Q^k A_j = \bigcup_{i=1}^m
((\Dk^k)_{ij} + A_i),\ \ \ j \le m. \] For any $i \le m$ and $a \in A_i$, we
can write
\[ a = Q^{-k} d_{i_1 i} + Q^{-2k} d_{i_2 i_1} + \cdots, \ \ \ \mbox{where
$d_{i_{n+1}i_n} \in (\Dk^k)_{i_{n+1} i_n}$.}\] Thus for any $i, j \le m$, $a \in A_i$, and $b \in A_j$,
\[|a - b| \le \sum_{n=1}^{\infty} \Vert Q^{-k}\Vert^n |d_{i_n i_{n-1}} - {d'}_{i_n i_{n-1}}| \le \frac{e^{(k)} \cdot \Vert Q^{-k}\Vert}{1 - \Vert Q^{-k}\Vert}.\]
If $(u,y,v)$ is an overlap where $u + T_i, v + T_j \in \Tk$ for
some $i,j \le m$, then
\[ (u+ A_i -y) \cap (v + A_j) \neq \emptyset.\]
Let $z \in (u+ A_i -y) \cap (v + A_j)$. Then $z-u+y \in A_i$ and
$z-v \in A_j$. So \be |u-y-v|  \le  \frac{e^{(k)} \cdot
\Vert Q^{-k}\Vert}{1 - \Vert Q^{-k}\Vert} = R. \nonumber \ee
\end{proof}

Similarly to the $k$-th iterated overlap, for each potential
overlap $\mathcal{O} = (u, y, v)$ in $\Tk$, we define {\em $k$-th
inflated potential overlap}
\begin{eqnarray*}
 {\Phi}^k \mathcal{O} = \{(u', Q^k y, v') \, :
  u' \in \Phi^k(u), v' \in \Phi^k(v), \ \mbox{and $(u',Q^ky,v')$ is a potential overlap} \}
\end{eqnarray*}
and the equivalence class of $\Phi^k \mathcal{O}$
\[ [{\Phi}^k \mathcal{O}] = \{[\mathcal{O}'] : \mathcal{O}' \in {\Phi}^k \mathcal{O} \}\,.\]





For the computation of overlap coincidence, it is important to
have the Meyer property of $\Xi(\Tk)$. The next theorem gives a
criterion on $Q$ for the Meyer property. A set of algebraic
integers $\Theta = \{\theta_1, \cdots, \theta_r \}$ is a {\em
Pisot family} if for any $1 \le j \le r$, every Galois conjugate
$\gamma$ of $\theta_j$ with $|\gamma| \ge 1$ is contained in
$\Theta$.

\begin{theorem} {\em \cite{LS2}} \label{PisotFamily-MeyerSet}
Let $\Tk$ be a self-affine tiling in $\R^d$ with a diagonalizable
expansion map $Q$. Suppose that all the eigenvalues of $Q$ are
algebraic conjugates with the same multiplicity. Then $\Xi(\Tk)$
is a Meyer set if and only if the set of all the eigenvalues of
$Q$ is a Pisot family.
\end{theorem}

Summarizing the results of this paper, we provide an algorithm to
determine the pure point spectrum of a substitution tiling
dynamical system.

\noindent\hrulefill \\
{\bf Algorithm} :  We assume that $\Tk$ is a self-affine tiling in
$\R^d$ with expansion map $Q$ for which $\Xi(\Tk)$ is a Meyer set
and $T_i = (A_i, i)$, $i \le m$, are prototiles such that
\[ QA_j = \bigcup_{i \le m} (\mathcal{D}_{ij} + A_i) \ \ \ \mbox{for $j \le m$}.\]

\begin{itemize}
\item {\bf Input:} $\Phi$ is an $m \times m$ matrix whose each
    entry is a set of functions from $\R^d$ to $\R^d$ such
    that $\Phi = (\Phi_{ij})$, where $\Phi_{ij} = \{f:x \to
    Qx+d, d \in \mathcal{D}_{ij}\}$, i.e. $\Phi$ is a MFS for
    $\Tk$.
\item {\bf Output:} True, if and only if $\Tk$ has pure point
    spectrum.
\end{itemize}

\vspace{2mm}

\begin{enumerate}
\item Find an initial point $x$ such that $x \in \Phi(x)$.
\item Find a basis $\{\alpha_1, \dots, \alpha_d \} \subset
    \R^d$ such that $\alpha_k \in \bigcup_{i \le m}
    ((\Phi^n(x))_i - (\Phi^n(x))_i)$ for some $n \in \Z_+$,
    for each $1 \le k \le d$.
\item For each $1 \le k \le d$, find all the potential
    overlaps $\mathcal{G}_{\alpha,0}$ which occur from the
    translation $\alpha_k$.
\item Find all the potential overlaps $\mathcal{G}$ which
    occur from the translations $Q^n \alpha_k$ with $1 \le k
    \le d$ and $n \in \Z_+$.
\item Find all the potential overlaps $\mathcal{G}_{coin}$
    which lead to coincidences within $\sharp \Gk$-iterations.
\item If $\rho(\Gk_{coin}) > \rho(\Gk \, \backslash \,
    \Gk_{coin})$, where $\rho(G)$ is the spectral radius of
    the graph $G$, return true. Else, return false.
\end{enumerate}
\vspace{-2.0mm}
 \noindent\hrulefill


\section{Computing coincidence} \label{Computing coincidence}

\noindent In the rest of the paper, we assume that $\Tk$ is a
self-affine tiling in $\R^d$ such that $\Xi(\Tk)$ is a Meyer set.
From (\ref{tile-subdiv}), for any $a_i \in A_i$ with $i \le m$, we
can get \be \label{new-tile-equation} Q (A_j - a_j) =
\bigcup_{i=1}^m (\Dk_{ij} - Q a_j + a_i +(A_i - a_i)) \ \ \
\mbox{for} \  j\le m. \ee We may consider new prototiles
\[\{T_1-a_1, \dots, T_m -a_m\}\] with new digit sets
\be \label{new-digit-sets} \Dk'_{ij} = \Dk_{ij} - Qa_j + a_i. \ee
Without loss of generality we can assume that for any $T_i = (A_i,
i) \in \Ak$, $i \le m$, \[ 0 \in A_i.\]

\subsection{Meyer sets}

Let $\Lam$ be a Meyer set and $[\Lam]$ be the Abelian group
generated by $\Lam$. Then $[\Lam]$ is finitely generated. So
$[\Lam] = \oplus_{i=1}^s \Z e_i$. We define $||\cdot|| : [\Lam]
\to \N$ such that $||\sum_{i=1}^s a_i e_i|| = \sum_{i=1}^s |a_i|$.
For each positive integer $n$, let
\[F(n):= \{u \in [\Lam] \, : \, ||u|| \le n\}.\]
Note that $F(n)$ is finite. Choose $h > 0$ such that every open
ball of radius $h$ in $\R^d$ meets at least one element in $\Lam$.
Since $\Lam - \Lam := \{x - y : x, y \in \Lam \}$ is uniformly
discrete from the Meyer property of $\Lam$, we let $L \in \Z_+$ be
an upper bound for the number of points in $\Lam - \Lam $ that can
lie in an open ball of radius $2h$. Let
\[ \ell:= {\mmax} \{||u|| \, : \, u \in \Lam - \Lam, |u| < 3 h \} \,.\]

\begin{prop}{\em \cite{Lag, RVM97}} \label{almostLatticeProperty}
Let $\Lam$ be a Meyer set. Then \[\Lam - \Lam \subset \Lam + F, \
\ \ \mbox{where $F= F(2 \ell(L-1))$}.\]
\end{prop}

\medskip

It is proved in \cite[Lemma\,A.8]{LMS2} that the number of
equivalence classes of overlaps for a tiling which has the Meyer
property is finite. We apply the same argument to get the number
of equivalence classes of potential overlaps for a tiling and give
an explicit upper bound for the number.

\begin{lem} \label{number-of-overlaps}
Let $\Tk$ be a self-affine tiling and $\Lb_{\Tk} = (\Lam_i)_{i \le
m}$ be the associated substitution Delone multi-colour set of
$\Tk$. Let $\Lam = \bigcup_{i \le m} \Lam_i$. Suppose that $\Lam$
is a Meyer set. The number of equivalence classes of potential
overlaps for $\Tk$ is less than or equal to $m^2 I$, where \[ I =
\#((\Lam + F + F + F) \cap B_R(0)),\] with $F = F(2 \ell (L -1))$
as in Prop.\,\ref{almostLatticeProperty}.
\end{lem}

\proof Let $(u, y, v)$ be a potential overlap in $\Tk$ for which
$u + T_i, v + T_j \in \Tk$. Then $|u - y - v| \le R$. Note that $u
- y - v \in (\Lam - \Lam) - (\Lam - \Lam)$. From
Prop.\,\ref{almostLatticeProperty},
\begin{eqnarray*}
(\Lam - \Lam) - (\Lam - \Lam) & \subset &  \Lam + F - (\Lam + F) \\
& \subset & \Lam + F + F + F\,.
\end{eqnarray*}
The equivalence classes of the potential overlaps are completely
determined by $i$, $j$ and the vector $u - y - v$. Thus the claim
follows. \qed

\medskip

\noindent {\bf Remark} 1. Let $\Lb = (\Lam_i)_{i \le m}$ be a
substitution Delone multi-colour set for which $\Lam = \cup_{i \le
m} \Lam_i$ is a lattice. It has been shown in \cite{Sing} that the
modular coincidence, which is equivalent to overlap coincidence in
lattice substitution Delone multi-colour sets, can be determined
within the exponential bound of $2^m-m-2$. Note that there are
only $m^2 I$ number of potential overlaps where $I = \#(\Lam \cap
B_R(0))$, since $\Lam - \Lam = \Lam$. Notice that $I$ is fairly
small number in the case of lattice substitution Delone
multi-colour set. Overlap coincidence of $\Tk_{\Lbs}$ can be
determined within ($m^2 I + 1$) number of iterations of each
potential overlap. This polynomial iteration bound is much smaller
than the exponential iteration bound given in \cite{Sing}.

2. We note from \cite[Th.\,4.14]{LS} that $\Lam$ is a Meyer set if
and only if $\Xi(\Tk)$ is a Meyer set in the self-affine tiling
$\Tk$.

\subsection{Coincidence and computation}

From now on, we assume that $\Tk$ is a self-affine tiling with an
expansion map $Q$ for which $\Xi(\Tk)$ is a Meyer set.

\medskip

For $\alpha \in \Xi(\Tk)$, define
\begin{eqnarray*}
\mathcal{E}_{\alpha} := \{(u, Q^n \alpha, v) \,: \,
 \mbox{$(u,Q^n \alpha,v)$ is overlap in $\Tk$, $n \in \N$ }\}.
\end{eqnarray*}
For any $n \in \Z_{\ge 0}$, define
\[D_{Q^n \alpha} := \Tk \cap (\Tk - Q^n \alpha) \]
and
\[\dens(D_{Q^n \alpha}) = \lim_{n \to \infty} \frac{\Vol(D_{Q^n \alpha} \cap B_n)}{\Vol(B_n)}\,.\]
\medskip

The following lemma is proved in \cite[Lemma\, A.9]{LMS2} with the
subdivision graph for overlaps.

\begin{lem}{\em \cite[Lemma\, A.9]{LMS2}} \label{equivalence-of-coincidence-for-alpha}
Let $\alpha \in \Xi(\Tk)$. The following are equivalent:
\begin{itemize}
\item[(i)] $\lim_{n \to \infty} \dens(D_{Q^n \alpha}) = 1$.
\item[(ii)] $1 - \dens(D_{Q^n \alpha}) \le b r^n$ for any $n
    \ge 1$, for some $b > 0$ and $r \in (0,1)$.
\item[(iii)] There exists $\ell \in \Z_+$ such that for any
    overlap $\mathcal{O}$ in $\mathcal{E}_{\alpha}$,
    ${\Phi}^{\ell} \mathcal{O}$ contains a coincidence.
\end{itemize}
\end{lem}

The next theorem is basically in \cite{soltil} and
\cite[Th.\,4.7]{LMS2}. We notice here that we only need to
consider the overlaps in $\mathcal{E}_{\alpha}$ for all $\alpha
\in \mathcal{B}$ to check the overlap coincidence of $\Tk$. We
rewrite the theorem in the form that we use here.

\begin{thm}{\em \cite{soltil}, \cite[Th.\,4.7]{LMS2}} \label{coincidence-check-on-finite-basis}
Let $\Tk$ be a self-affine tiling for which $\Xi(\Tk)$ is a Meyer
set. Then for any $\alpha \in \mathcal{B}$ and any overlap
$\mathcal{O} \in \mathcal{E}_{\alpha}$, $\Phi^{\ell} \mathcal{O}$
contains a coincidence for some $\ell = \ell(\mathcal{O}) \in \N$
if and only if $\Tk$ admits an overlap coincidence.
\end{thm}

\proof We only prove the sufficiency direction, since the other
direction is clear. Suppose that for any $\alpha \in \mathcal{B}$
and any overlap $\mathcal{O} \in \mathcal{E}_{\alpha}$,
$\Phi^{\ell} \mathcal{O}$ contains a coincidence. From the
argument of \cite[Lemma\,A.9]{LMS2}, for some $b > 0$ and $r \in
(0,1)$
\[ 1 - \dens(D_{Q^n \alpha}) \le b r^n \ \ \ \mbox{for any $n \in \N$}\,.\]
Hence
\[ \sum_{n=0}^{\infty} (1 - \dens(D_{Q^n \alpha})) < \infty \,.\]
Since $\mathcal{B}$ forms a basis for $\R^d$, by
\cite[Thm.\,6.1]{soltil} the dynamical system of Delone
multi-colour set has pure point spectrum. By \cite[Thm.\,4.7 and
Lemma\,A.9]{LMS2}, $\Tk$ admits an overlap coincidence. \qed

\medskip

\noindent
In order to find first all equivalent classes of potential
overlaps which occur from the translations of $\alpha_i$ for any
$1 \le i \le d$, we want to know how much region of the
intersection of $\Tk$ and $\Tk - \alpha_i$, $1 \le i \le d$, we
have to look. We use same notations for points with colour in
$\Lb_{\Tk}$ and points in $\R^d$. This should not cause any
confusion.

Let
\[ \Theta = \{[\Pb] : \Pb = \{y, z\} \subset \Lb_{\Tk} \ \ \mbox{satisfies}
\ \ |y - z| < R + \Vert Q^{-k}\Vert \alpha_{max} \}.
\]
Let
\[\mathcal{J}(\Gb) := \{[\Pb] \in \Theta: \Pb \subset \Gb \},
\ \ \ \mbox{where $\Gb \subset \Lb_{\Tk}$}\,.\]

\begin{lem}
\label{init0}
If $\Jk(\Phi^N(\xi)) = \Jk(\Phi^{N+k}(\xi))$ for some $N \in
\Z_+$, then
\[\Jk(\Phi^N(\xi)) = \Jk(\Lb_{\Tk}) \,.\]
\end{lem}

\begin{proof} Let $\Pb$ be a cluster in $\Phi^{N+k+1}(\xi)$ such that
$[\Pb] \in \Theta$. There must be a cluster $\Pb' = \{y, z\}
\subset \Phi^{N+1}(\xi)$ satisfying $\Pb \subset \Phi^k(\Pb')$. We
claim that $[\Pb'] \in \Theta$. We only need to show that $|y - z|
\le R + \Vert Q^{-k}\Vert \alpha_{max}$. Suppose that $|y - z|> R + \Vert Q^{-k}\Vert \alpha_{max}$. Then
\[|Q^k y - Q^k z| \ge  \frac{1}{\Vert Q^{-k}\Vert} |y -z| > \frac{1}{\Vert Q^{-k}\Vert} (R + \Vert Q^{-k}\Vert \alpha_{max})
= \frac{R}{\Vert Q^{-k}\Vert} + \alpha_{max}.
\] For any $y' \in \Phi^k(y)$ and
$z' \in \Phi^k(z)$, $y' = Q^k y +d_1$ and
 $z'=Q^k z + d_2$ for some $d_1, d_2 \in \cup_{i,j \le m} {(\mathcal{D}^k)}_{ij}$. Thus
\begin{eqnarray*}
|y' -z'| &=& |Q^k y - Q^k z + d_1 - d_2| \ge |Q^k y - Q^k z| -
e^{(k)} \\
& > & \frac{R}{\Vert Q^{-k}\Vert} + \alpha_{max} -e^{(k)} = R +
\alpha_{max} \ \ \ \mbox{from (\ref{formula-diam})} \\
& > & R + \Vert Q^{-k}\Vert \alpha_{max}.
\end{eqnarray*} It contradicts to
the choice of $\Pb$. Hence $[\Pb'] \in \Theta$. From the
assumption, note that
\[\Jk(\Phi^N(\xi)) = \Jk(\Phi^{N+1}(\xi)) =
\cdots = \Jk(\Phi^{N+k}(\xi))\,.\] So there exists a cluster
$\Pb''$ in $\Phi^{N}(\xi)$ which is equivalent to $\Pb'$. Then
$\Phi^k(\Pb'')$ contains a cluster which is equivalent to the
cluster $\Pb$. Thus
\[  \Jk(\Phi^{N+k}(\xi)) = \Jk(\Phi^{N+k+1}(\xi)) \,.\]
Hence
\[ \mathcal{J}(\Phi^N(\xi)) = \mathcal{J}(\lim_{n \to \infty} \Phi^n(\xi))\,.\]
By the repetitivity of $\Lb_{\Tk}$, all the clusters in
$\Lb_{\Tk}$ whose equivalent classes are in $\Theta$ should occur
in $\mathcal{J}(\lim_{n \to \infty} \Phi^n(\xi))$. Therefore
$\mathcal{J}(\Phi^N(\xi)) = \mathcal{J}(\Lb_{\Tk})$.
\end{proof}

\begin{lem} \label{initial-potential-overlap}
If $\Jk(\Phi^N(\xi)) = \Jk(\Lb_{\Tk})$, then for each $\alpha \in
\mathcal{B}$,
\[\mathcal{G}_{\alpha, 0} := \{[(y, \alpha, z)] : y, z \in \Phi^{N+k}(\xi) \
\mbox{and $|y - \alpha -z| < R$} \}\] contains all the different equivalence classes
of potential overlaps which can occur from the translation of
$\alpha$.
\end{lem}

\begin{proof}
If $\{y, z\} \subset \Lb_{\Tk}$ such that $|y - \alpha - z| < R$
for $\alpha \in \mathcal{B}$, there exist $u, v \in \Lb_{\Tk}$
such that $y \in \Phi^k(u)$ and $z \in \Phi^k(v)$. Let $y = Q^k u
+ d_1$ and $z = Q^k v + d_2$, where $d_1, d_2 \in \cup_{i,j \le m}
(\mathcal{D}^k)_{ij}$. Let $\Pb = \{u, v\}$. We claim that $[\Pb]
\in \mathcal{J}(\Phi^{N}(\xi))$. Suppose that $|u - v| > R + \Vert
Q^{-k}\Vert \alpha_{max}$. Then
\begin{eqnarray*}
|y -z| &=& |Q^k u - Q^k v + d_1
-d_2| \ge |Q^k(u-v)| - e^{(k)} \\
& \ge & \frac{1}{\Vert Q^{-k}\Vert}|u-v| - e^{(k)} >
\frac{ R}{\Vert Q^{-k}\Vert} + \alpha_{max} - e^{(k)} = R + \alpha_{max}
\end{eqnarray*}
It contradicts to the choice of $\{y, z\} \subset \Lb_{\Tk}$.
Since $\Jk(\Phi^N(\xi)) = \Jk(\Lb_{\Tk})$, $[(y, \alpha, z)] \in
\mathcal{G}_{\alpha, 0}$.
\end{proof}

\begin{lem}
If $\mathcal{J}(\Phi^N(\xi)) = \mathcal{J}(\Lb_{\Tk})$ for some $N
\in \Z_+$, then for each $n \in \Z_{\ge 0}$ and $\alpha \in
\mathcal{B}$,
\[ \{[(y, Q^n \alpha, z)] : y, z \in \Phi^{N+k+n}(\xi) \
\mbox{and $(y, Q^n \alpha, z)$ is a potential overlap} \} \]
contains all the different equivalence classes of potential
overlaps which occur from the translation of $Q^n \alpha$.
\end{lem}

\proof We argue this by induction. Note that when $n=0$, the claim
is true. Suppose that it is true for $n=i$, $i \in \Z_+$. Consider
$n=i+1$. Let $\mathcal{O}$ be a potential overlap which occurs
from the translation of $Q^{i+1} \alpha$. Then there exist $y, z
\in \Phi^{N+i+s}(\xi) \subset \Lb$ for some $s \in \N$ such that
\[\mathcal{O} = [(y, Q^{i+1} \alpha, z)].\] Then there exists a
potential overlap $(u, \alpha, v)$ with $u, v \in
\Phi^{N+i+s-1}(\xi)$ such that
\[(y, Q^{i+1} \alpha, z) \in \Phi(u, Q^i \alpha, v)\,.\]
But we know that there exists a potential overlap $(u', Q^i
\alpha, v')$ with $u', v' \in \Phi^{N+k+i} (\xi)$ which is
equivalent to a potential overlap $(u, Q^i \alpha, v)$ by the
assumption. Thus there exists an equivalent potential overlap
$(y', Q^{i+1} \alpha, z')$ to $(y, Q^{i+1} \alpha, z)$ which is
contained in $\Phi(u', Q^i \alpha, v')$. Note that
\[ y', z' \in \Phi^{N+k+i+1}(\xi)\,.\]
Thus
\[ \{[(y, Q^{i+1} \alpha, z)] : y, z \in \Phi^{N+k+i+1}(\xi)\} \]
contains all the different types of equivalent potential overlaps
which occur from the translation $Q^{i+1} \alpha$. Thus the claim
is proved. \qed

\medskip

For any $\alpha \in \mathcal{B}$ and any $M \in \Z_{\ge 0}$,
define
\[\Gk_{\alpha, M} := \bigcup_{0 \le n \le M} \{[(y, Q^n \alpha, z)] :
y, z \in \Phi^{N+k+n}(\xi)  \
\mbox{and $(y, Q^n \alpha, z)$ is a potential overlap} \}\,\]
\[\Gk_{\alpha} := \bigcup_{M \in \Z_{\ge 0}} \Gk_{\alpha, M} \ \ \mbox{and} \ \
\Gk = \bigcup_{\alpha \in \mathcal{B}} \Gk_{\alpha}.\]

\begin{lem} \label{finding-all-overlaps}
Let $\alpha \in \mathcal{B}$. If $\Gk_{\alpha, M} = \Gk_{\alpha,
M+1}$ for some $M (= M_{\alpha}) \in \Z_{\ge 0}$, then
\[  \Gk_{\alpha, M} = \Gk_{\alpha}.\]
\end{lem}

\proof Let $[(y, Q^{n} \alpha, z)] \in \Gk_{\alpha, M+2}$, where
$0 \le n \le M+2$. Then there exists a potential overlap $(y',
Q^{n-1} \alpha, z')$, with $y', z' \in \Phi^{N+k+M+1}(\xi)$, such
that
\[(y, Q^{n} \alpha, z) \in \Phi(y', Q^{n-1} \alpha, z').\]
Since $\Gk_{\alpha, M} = \Gk_{\alpha, M+1}$, $(y', Q^{n-1} \alpha,
z')$ is equivalent to $(y'', Q^{n'} \alpha, z'')$ for some $y'',
z'' \in \Phi^{N+k+{n'}}(\xi)$ where $0 \le n' \le M$. So
\begin{eqnarray*}
[\Phi (y', Q^{n-1} \alpha, z')] = [\Phi (y'', Q^{n'} \alpha, z'')] \,.
\end{eqnarray*}
Thus $[(y, Q^{n} \alpha, z)] \in \Gk_{\alpha, M+1}$. \qed

\medskip
Let $\mathcal{H}$ be the set of all equivalent classes of overlaps
in $\Tk$ and $\mathcal{G}_{coin}$ be the set of all equivalence
classes of overlaps in $\Tk$ which lead to coincidence after some
iterations. Note that \[\mathcal{G}_{coin} \subset \mathcal{H}
\subset \mathcal{G}\] and for each $\alpha \in \mathcal{B}$, $
M_{\alpha} \le m^2 I$, since there are at most $m^2 I$ equivalence
classes of potential overlaps in $\Tk$ by the
Lemma\,\ref{number-of-overlaps}.

\medskip
\begin{lem}
If an overlap in $\Tk$ has a coincidence after some iterations, it
should happen before $\sharp \mathcal{G}$ number of iterations,
i.e.
\begin{eqnarray*}
 \mathcal{G}_{coin} & = & \{[(u, Q^n \alpha, v)] \in
\mathcal{G} \, : \, \Phi_{\ell i}^t (u- Q^n \alpha) \cap \Phi_{\ell
j}^t (v) \neq \emptyset \\
&& \ \ \ \mbox{for some $1 \le \ell \le m$ and $0 \le t < \sharp\mathcal{G}$,
where $u + T_i, v + T_j \in \Tk$}\}.
\end{eqnarray*}
\end{lem}

\proof Note that there are at most $\sharp\mathcal{G}$ number of
equivalence classes of potential overlaps in $\Tk$. For any
overlap $\mathcal{O}$, if coincidence does not occur in
$\Phi^t(\mathcal{O})$ for some $0 \le t \le \sharp\mathcal{G}$,
coincidence will never occur in $\Phi^n(\mathcal{O})$ for any $n
\in \Z_{\ge 0}$. Since $Q$ is an expansive map, it is sufficient
to check for $0 \le t < \sharp\mathcal{G}$.

\section{The potential overlaps that are not real overlaps}
\label{The potential overlaps that are not real overlaps}

%


\noindent We aim to prove the following Theorem\,\ref{main} in
this section. The algorithm given by this theorem is quite simple
and easy to implement and applies to all self-affine tilings
whenever the expansion maps $Q$ and digit sets $\Dk_{ij}$, which
define the self-affine tilings, are given.

\medskip

In the sequel, we construct a {\em graph with multiplicities}
viewing potential overlaps in $\mathcal{G}$ as vertices and define
multiple edges by counting the vertices in the inflated potential
overlaps. Hereafter we deal with the representatives of
equivalence classes of potential overlaps in $\mathcal{G}$. Let
$(u, y, v)$ be a potential overlap, where $u+T_i, v+T_j \in \Tk$,
$T_i = (A_i,i)$ and $T_j=(A_j,j)$. Inflating the corresponding
tiles in the potential overlap $(u, y, v)$ and intersecting them,
we observe
\begin{eqnarray} \label{intersection-boundaryTouchingTiles}
\lefteqn{Q(u+A_i-y) \cap Q(v+A_j)} \nonumber \\
 & = & \bigcup_{k=1}^m (A_{k}+\Dk_{ki}+ Q u -Q y) \cap
\bigcup_{\ell=1}^m (A_{\ell}+\Dk_{\ell j}+Q v)\\
& = & \bigcup_{k=1}^m \bigcup_{\ell=1}^m \bigcup_{d_{ki} \in \Dk_{ki}} \bigcup_{d_{\ell j} \in \Dk_{\ell j}}
(((A_{k}+d_{ki}-d_{\ell j}+Q u-Q y-Q v) \cap A_{\ell}) + d_{\ell j} + Q v). \nonumber
\end{eqnarray}
The equivalence class $[(u,y,v)]$ can be viewed as an element
$(i,y,j)$ where $1 \le i,j \le m$ and $z = u-y-v$ with $|z| \le R$
where $R$ is as defined in (\ref{formula-diam}). We define the
multiple edge
\begin{equation}
\label{transition}
(i, z,j) \stackrel{e}{\rightarrow} (k, z', \ell)
\end{equation}
if $z'=d_{ki}-d_{\ell j}+Q x_i -Q y -Q x_j$ with $|z'|\le R$ for
$d_{ki}\in \mathcal{D}_{ki}$ and $d_{\ell j}\in
\mathcal{D}_{\ell j}$, where the multiplicities of the edge is
given by ${\#}\{(d_{ki},d_{\ell j}) \in \Dk_{ki} \times
\Dk_{\ell j}\ |\ z'= d_{ki}-d_{j \ell}+Q u - Q y - Q v \}$.
Keeping the multiplicity in the graph is essential to distinguish
real overlaps from potential overlaps that are not. Recall that
$(i,y,j)$ is a coincidence if $i=j$ and $y=0$. We consider
$\mathcal{G}_{coin}$ as the induced graph of $\mathcal{G}$ to the
vertices which have a path leading to a coincidence. Also we
define $\mathcal{G}_{res}$ by the induced graph generated by the
complement of such set from $\mathcal{G}$, i.e. $\mathcal{G}_{res}
= \mathcal{G}\, \backslash \, \mathcal{G}_{coin}$. For any graph
$G$, we denote by $\rho(G)$ the spectral radius of the graph $G$.

\begin{thm} \label{main} Let $\Tk$ be a self-affine tiling for which
$\Xi(\Tk)$ is a Meyer set. Then the following are equivalent;
\begin{itemize}
\item[(i)] $\Tk$ admits an overlap coincidence.
\item[(ii)]
    $\rho(\mathcal{G}_{coin})>\rho(\mathcal{G}_{res})$.
\end{itemize}
\end{thm}

\medskip

The potential overlaps $(u,y,v)$ can be divided into three cases.
\begin{itemize}
\item No intersection overlap: $(u + A_i - y) \cap (v+A_j) =
    \emptyset$,
\item Boundary touching overlap: $u+A_i-y$ and $v+A_j$ are
    just touching at their boundaries,
\item Real overlap: $(u+ A_i-y)^{\circ}\cap {(v+A_j)}^{\circ}
    \neq \emptyset$.
\end{itemize}

If $(u+A_i -y) \cap (v+A_j)$ is empty, then the distance between
two tiles $u+T_i-y$ and $v+T_j$ becomes larger by the iterations
of the tile-substitution $\Omega$. Therefore this potential
overlap does not produce an infinite walk on the graph of
potential overlaps by the iterations of $\Omega$. However, when
they are touching at their boundaries, this gives infinite walks
on the graph and it may contribute to the number of possible paths
and consequently to the spectral radius by repeated inflation. Our
task is to prove that this contribution is small so that we can
distinguish them from real overlaps.

\medskip

Let $(V, \Gamma)$ be a directed graph with a vertex set $V = \{1,
\dots, M \}$ and an edge set $\Gamma$. We call $\{f_e : e \in
\Gamma\}$, a collection of contractions $f_e: \R^d \to \R^d$, a
{\em graph-directed iterated function system (GIFS)}. Let
$\Gamma_{k\ell}$ be the set of edges from vertex $k$ to $\ell$,
then there are unique non-empty compact sets $\{E_k\}_{k=1}^N$
satisfying
\begin{equation}\label{GIFS}
E_k=\bigcup_{\ell=1}^M \bigcup_{e\in \Gamma_{k\ell}} f_e(E_{\ell}),  \ \ \ \mbox{for $k \le M$}
\end{equation}
(see \cite{Mauldin-Williams:88}). We say that (\ref{GIFS})
satisfies the {\em open set condition} (OSC) if there are open
sets $U_k$ so that
$$
\bigcup_{\ell=1}^M \bigcup_{e\in \Gamma_{k\ell}} f_e(U_{\ell})\subset U_k, \ \ \  \mbox{for $k \le M$}
$$
and the left side is a disjoint union. Further if $U_k\cap E_k\neq
\emptyset$ for all $k \le M$, then we say that the GIFS satisfies
the {\em strong open set condition} (SOSC).
\medskip

We observe from (\ref{intersection-boundaryTouchingTiles})
\begin{eqnarray} \label{equation-potentialOverlap-graph}
\lefteqn{(A_i + u -y - v) \cap A_j = } \nonumber \\
& & \bigcup_{k \le m} \bigcup_{\ell \le m}
\bigcup_{d_{ki} \in \Dk_{ki}} \bigcup_{d_{\ell j} \in \Dk_{\ell j}}
Q^{-1} \left( ((A_{k}+d_{ki}-d_{\ell j}+Q u -Q y - Q v) \cap A_{\ell}) +d_{\ell j} \right).
\end{eqnarray}
If $(k, z', \ell)$, where $z' = d_{ki}-d_{\ell j}+Q u-Q y - Q v$,
is not a potential overlap, we discard it from
(\ref{equation-potentialOverlap-graph}).

Let $M$ be the number of elements in $\mathcal{G}_{res}$. Now we
construct a graph for $\mathcal{G}_{res}$ identifying the
potential overlaps in $\mathcal{G}_{res}$ with the numbers in
$\{1, \cdots, M\}$. In the graph $\mathcal{G}_{res}$, if vertices
have no outgoing edges, we can remove them from
$\mathcal{G}_{res}$ successively. For each potential overlap
$(i,v,j)$ which corresponds to a vertex $k \le M$, let \be
\label{define-boundaryIntersection} E_k :=(A_i + z) \cap A_j, \ \
\ \mbox{where $z=u-y-v$}.\ee Let $\Gamma_{k \ell}$ be the set of
edges from $k$ to $\ell$ where $1 \le k, \ell \le M$. From
(\ref{equation-potentialOverlap-graph}), we notice that $E_k$'s
satisfy GIFS (\ref{GIFS}), where $f_e(x) = Q^{-1}(x + d_e)$ and
$d_e \in \mathcal{D}_{\ell j}$ for some $\ell \le m$. Denote by
$\Gamma_{k\ell}^n$ the set of paths of length $n$ from $k$ to
$\ell$ and for $I=e_1\dots e_n\in \Gamma_{k\ell}^n$ we put
$d_I=\sum_{i=1}^n Q^{n-i} d_{e_i}$.

\begin{rem} \label{subset-of-boundary}
If $\mathcal{G}_{res}$ does not contain the vertices of real
overlaps, each $E_k$ with $k \le M$ is a subset of $\partial
A_{j}$ for some tile $T_{j} = (A_{j}, j) \in \mathcal{A}$.
\end{rem}

\medskip

We use the recent development by He and Lau \cite{He-Lau:08} which
slightly modifies the Hausdorff measure. They introduced a new
type of gauge function, called {\em pseudo norm}
$w:\R^d\rightarrow \R_+$ corresponding to $Q$ having key
properties:
\begin{equation}
\label{contraction}
w(Qx)=|\det(Q)|^{1/d} w(x)
\end{equation}
and
\begin{equation}
\label{WeakNorm}
w(x+y)\le c(w) \max(w(x),w(y))
\end{equation}
for some positive constant $c(w)$. This $w$ induces the same
topology as Euclidean norm. By $w$, they modified the definition
of Hausdorff measure by: for an open set $U \subset \R^d$, a
subset $X \subset \R^d$ and $s, \delta \in \R_+$,
$$
\diam_w(U)=\sup_{x,y\in U} w(x-y),
$$
$$
\mathcal{H}^{s,\delta}_w (X)=\left. \inf_{X\subset \cup_i U_i}
\left\{\sum_i \diam_w(U_i)^s \ \right|\
\diam_w(U_i)<\delta \right\}
$$
and
$$
\mathcal{H}^{s}_w(X)=\lim_{\delta\downarrow 0} \mathcal{H}^{s,\delta}_w (X).
$$
Our new Hausdorff dimension is defined by
$$\dim_H^w(X)=\sup\{s |
\mathcal{H}_w^s(X)=\infty\} =\inf\{s | \mathcal{H}_w^s(X)=0\}.$$
Using the pseudo norm $w$, one can treat self-affine attractors
almost as easy as self-similar ones.

\medskip
To prove Theorem\, \ref{main},
we need the next lemma of Luo-Yang
\cite{Luo-Yang:09}.
This generalizes
a result in \cite{He-Lau:08} and its proof basically
follows from the idea of Schief \cite{Schief:94}, but
using pseudo norm instead of Euclidean norm.
Note that strong connectedness of GIFS is essential.

\begin{lem} \cite[Th.\,1.1]{Luo-Yang:09} \label{OSC}
Assume that the GIFS is strongly connected. Then following conditions are equivalent:
\begin{enumerate}
\item $\{d_I\ |\ I\in \Gamma_{k\ell}^n\}$ give distinct
    $\#(\Gamma_{k\ell}^n)$ points whose distance between two
    points has a uniform lower bound $r>0$ for all $k,\ell \le
    M$ and $n \ge 1$.
\item The GIFS satisfies strong open set condition (SOSC).
\end{enumerate}
\end{lem}

\noindent{\it Proof of Theorem\,\ref{main}} (ii) $\Rightarrow$
(i). Consider an overlap $(u, y, v)$ where $u + T_i, v + T_j \in
\Tk$. Applying $Q^n$ to the overlapping part of the overlap, we
have
$$
\mu_d(Q^n((u + A_i - y)\cap (v +A_j)))=|\det(Q)|^n \mu_d((u +A_i - y)\cap (v +A_j))
$$
where $\mu_d$ is the $d$-dim Lebesgue measure. We know $|\det Q|=
\beta$ where $\beta$ is the Perron Frobenius root of substitution
matrix $(\#(\Dk_{ij}))$ (see \cite{lawa}). Since there are only
finitely many overlaps up to translations, there exist $r>0$ and
$R>0$ such that $(u + A_i- y)^{\circ} \cap (v + A_j)^{\circ}$
contains a ball of radius $r$ and is surrounded by a ball of
radius $R$. After $n$-iteration of inflation, the number of
potential overlaps $K_n$ generated from $(u, y, v)$ is estimated:
$$
c_1 \beta^n \le K_n \le c_2 \beta^n
$$
with some positive constants $c_1$ and $c_2$. Each real overlap
gives this growth of potential overlaps. It implies that
$\mathcal{G}_{res}$ cannot contain any real overlap (Recall that
we are taking into account the multiplicities of overlap growth).
This shows the claim.

(i) $\Rightarrow$ (ii). We show that if all overlaps lead to a
coincidence then $\mathcal{G}_{res}$ cannot have a spectral radius
$\beta = |\det Q|(= \rho(\mathcal{G}_{coin}))$. By the assumption,
$\mathcal{G}_{res}$ does not contain overlaps. So from
Remark\,\ref{subset-of-boundary}, $Y=\bigcup_{k=1}^M E_k$, where
$E_k$'s are defined as in (\ref{define-boundaryIntersection}), is
the subset of the union of boundaries of tiles. By Lemma \ref{OSC}
the GIFS satisfies OSC because the uniform discreteness condition
(1) of the Lemma\,\ref{OSC} automatically follows from the fact
that $(\Dk^n)_{ij}$'s are uniformly discrete for any $i,j \in m$
and $n \ge 1$(see \cite{lawa}).

We follow Mauldin-Williams \cite{Mauldin-Williams:88} to compute a
new Hausdorff dimension using pseudo norm $w$ instead of Euclidean
norm.
Let
\[ s = d \log \gamma/\log \beta, \]
where $\gamma=\rho(\mathcal{G}_{res})$ and $\beta = |\mbox{det}
Q|$. We study the value $\mathcal{H}_{w}^s(Y)$. First, assuming
strong connectedness of GIFS and OSC, we show
$0<\mathcal{H}_{w}^{s}(Y)<\infty$ by using standard mass
distribution principle (c.f. Theorem 1.2 in \cite{Luo-Yang:09}).
Second we use a simple fact: an infinite path on GIFS must
eventually fall into a single strongly connected component. Thus
for GIFS without strong connectedness, we classify infinite walks
on GIFS by the prefixes before they fall into the last strongly
connected components. This gives an expression of an attractor of
general GIFS as a countable union of contracted images of
attractors which belong to strongly connected components. In this
way we can show the Hausdorff measure $\mathcal{H}_{w}^{s}$ is
positive and $\sigma$-finite, by applying Lemma \ref{OSC} to each
strongly connected component. This shows the new Hausdorff
dimension of $Y$ with respect to the pseudo norm $w$
$$
\dim^w_H(Y)= s.
$$
Notice that the value $\mathcal{H}_{w}^{s}(Y)$ can be infinite
since we do not know that our GIFS is strongly connected (c.f.
\cite[Th.\,4]{Mauldin-Williams:88}).

Now if $s=d$, then $\mathcal{H}_{w}^{s}$ is a translationally
invariant Borel regular measure having positive value for any open
sets because the pseudo norm is comparable with Euclidean norm
(\cite[Prop.\,2.4]{He-Lau:08}). Therefore $\mathcal{H}_{w}^s$ must
be a constant multiple of the $d$-dimensional Lebesgue measure, by
the uniqueness of Haar measure. But this is impossible because the
$d$-dimensional Lebesgue measure of the boundary of self-affine
tiles must be $0$(see \cite{Praggastis:99}). This shows $s=d\log
\gamma/\log \beta<d$ which completes the proof. \qed

\medskip

The following conjecture is a folklore. It is mentioned as an open
problem in \cite{Solomyak:08} from personal communication with M.
Urba\'nski.
\begin{conj} \label{folk}
For $d$-dimensional non-polygonal self-affine tiling $\Tk$, each
tile $T=(A,i)$ satisfies \[ d-1<\dim_H( \partial A)<d . \]
\end{conj}

We partially solve a version of this
conjecture in the following Theorem\,\ref{boundary}.
Indeed if the matrix $Q$ gives similitudes, this settles the right
inequality of Conjecture \ref{folk}.

\begin{thm}\label{boundary}
For $d$-dimensional self-affine tiling $\Tk$, each tile $T=(A,i)$
satisfies \[ \dim_H^w (\partial A)<d. \]
\end{thm}

\begin{proof}
We consider a collection of all pairs of tiles in $\Tk$ whose
boundaries are touching. As in (\ref{transition}) and
(\ref{equation-potentialOverlap-graph}), we get a new GIFS which
is defined on this collection. Applying the same argument as in
Theorem\,\ref{main}, we get $s=d\log \gamma/\log \beta<d$ which
shows the claim.
\end{proof}

\section{Examples} \label{Examples}

We implemented Mathematica programs which perform our algorithm to
check the overlap coincidence for self-affine tilings. Readers can
get the Mathematica programs in the following website.
\begin{verbatim}
http://mathweb.sc.niigata-u.ac.jp/~akiyama/Research1.html
\end{verbatim}
For a given expanding matrix $Q$ and digit sets $\Dk_{ij}$ of a
self-affine tiling which has the Meyer property, the program gives
outputs $\rho(G_{coin})$ and $\rho(G_{res})$. By Theorem
\ref{main}, we can determine whether it satisfies overlap
coincidence or not. If the tiling does not satisfy the Meyer
property, it may not stop, or stop but produce incorrect outputs.

In actual computation, it is the bottleneck of the program to find
all initial potential overlaps for Lemma\, \ref{init0} and
\ref{initial-potential-overlap}. So we use two major tricks in the
program to make the computation fast. First, we translate the
digits $\Dk_{ij}$ to $\Dk'_{ij}$ as shown in
(\ref{new-digit-sets}) such that the number of potential overlaps
and $e^{(k)}$ are small. The size of $e^{(k)}$ is significant in
the computation of collecting all the initial potential overlaps
in Lemma\,\ref{init0}. To make $e^{(k)}$ small, we obtain some
number of points in $A_i$ using the tile equation
(\ref{tile-subdiv}) and choose $a_i$ among them which is located
closest to their centroid. Then we shift tiles $A_i$ to $A_i-a_i$.
Second, in order to get all the initial potential overlaps
$\mathcal{G}_{\alpha,0}$ in
Lemma\,\ref{initial-potential-overlap}, we try to find a fine
lattice in $\R^d$ such that we make an embedding of an iterated
point set into the lattice taking the closest lattice point for
each point of the set. Using the lattice, we can easily compute
the candidates of initial potential overlaps, which is much faster
than dealing with the original point set. We list selected
examples of our computation below.

\medskip

For $1$-dimension substitution sequences, we can obtain
self-similar tilings by suspension which associates to each letter
the interval whose length is each entry of a left eigenvector of
the incidence matrix. Pure point spectrum for the $\Z$-action on a
substitution sequence dynamical system is equivalent to pure point
spectrum for the $\R$-action on its suspension tiling dynamical
system \cite{Clark-Sadun}. The following example shows how to
obtain a tile substitution when a symbolic letter substitution is
given. In this case, we give a separate computational program in
the above website to check overlap coincidence directly.

\begin{ex} \label{Period-doubling}
{\em We consider a period-doubling substitution $0 \to 01$ and $1
\to 00$. Then the substitution matrix is $\begin{pmatrix} 1 & 2 \\
1 & 0
\end{pmatrix}$ and $(1,1)$ is a left eigenvector of PF-eigenvalue $2$.
Giving length $1$ to each letter, we can consider the following
tile equation
\begin{eqnarray*}
2 A_1&=&A_1 \cup (A_2 + 1)\\
2 A_2&=&A_1 \cup (A_1 + 1)
\end{eqnarray*}
and get that its suspension tiling has overlap coincidence.

We have computed a substitution $ 0 \to 051000$, $1 \to 324100$,
$2 \to 24100$, $3 \to 324333$, $4 \to 051433$, $5 \to 51433$ given
in \cite[Ex.\,5.3]{BBK} and a substitution $0 \to 03$, $1 \to 0$,
$2 \to 21$, $3 \to 2$ [personal communication from B. Sing]. They
both do not give overlap coincidence (see \cite[Sec.\,6c.3]{Sing}
for alternative constructions of no pure point reducible Pisot
substitution). }
\end{ex}

\medskip

We show the results of the following examples in
Table\,\ref{Results}.

\begin{ex} \label{Fibonnaci}
{\em Fibonacci substitution tiling is a well-known $1$-dimension
tiling. The tile equation can be given by
\begin{eqnarray*}
\alpha A_1&=&A_1 \cup (A_2+1)\\
\alpha A_2&=&A_1
\end{eqnarray*}
where $\alpha^2-\alpha-1=0$. The corresponding MFS is $\Phi =
\begin{pmatrix} \{f_1\} & \{f_1\} \\ \{f_2\} & \emptyset
\end{pmatrix}$ where $f_1(x)=\alpha x$ and $f_2(x)= \alpha x +
1$.}
\end{ex}

\begin{ex}
\label{Dekking} {\em Dekking in \cite{Deka, Dekb} constructed
self-similar tilings from endomorphisms of a free group. Kenyon
extended this idea in \cite[\S 6]{Ken}. We examined Example 7.5 in
\cite{soltil} derived by this method (see Fig \ref{Ex7_5}). It is
known that the corresponding tiling dynamical system is not weakly
mixing and has a large discrete part in the spectrum. We check
that the dynamical system has pure point spectrum. The tile
equation is
\begin{eqnarray*}
\alpha A_1&=&A_2\\
\alpha A_2&=&(A_2-1-\alpha) \cup (A_3-1)\\
\alpha A_3&=&A_1-1
\end{eqnarray*}
with $\alpha \approx 0.341164 + i 1.16154 $ which is a root of the
polynomial $x^3+x+1$.

We identify  $\C$ with $\R^2$ to simplify the notation; the
multiplication of $\alpha$ in $\C$ is expressed by the
multiplication of a matrix $Q=\begin{pmatrix} a & - b
\\ b & a \end{pmatrix}$ in $\R^2$, where $\alpha = a+ bi$. The Delone multi-colour set is
given by:
\begin{eqnarray*}
\Lambda_1&=&\alpha \Lambda_3 -1 \\
\Lambda_2&=&\alpha \Lambda_1 \cup (\alpha \Lambda_2 -1-\alpha) \\
\Lambda_3&=&\alpha \Lambda_2 -1 \,.
\end{eqnarray*}
We take a basis $B$ of translation vectors $\{1+\alpha,
\alpha+\alpha^2\} \subset \Lambda_2-\Lambda_2 \subset \Xi(\Tk)$.
In Table\,\ref{Results}, we write this choice $\Lambda_2$ as
Colour 2. The MFS is
\[\Phi =\begin{pmatrix} \emptyset &
\emptyset& \{f_3\}
\\ \{f_1\} & \{f_2\} & \emptyset \\ \emptyset& \{f_3\} &
\emptyset
\end{pmatrix}\]
where $f_1 = \alpha x$, $f_2 = \alpha x -1 - \alpha$ and $f_3 =
\alpha x -1$. We obtain $\# \mathcal{G}_{coin}=15$, $\#
\mathcal{G}_{res}=24$\footnote{This number depends on other
parameters we choose for computation.}, $\rho(\mathcal{G}_{coin})
\approx 1.46557$ and $\rho(\mathcal{G}_{res}) \approx 1.32472$.
This shows overlap coincidence and therefore the tiling dynamical
system associated with this tiling has pure point spectrum. Since
this case is self-similar, the Hausdorff dimension w.r.t. the
pseudo norm coincides with the usual Hausdorff dimension. So the
Hausdorff dimension of the boundary of each tile\footnote{In this
example, the graph $\mathcal{G}_{res}$ is weakly connected and has
only one strongly connected component of spectral radius greater
than one. Therefore the boundary of each tile has the same
dimension. The same holds for all examples in this paper.} is
$2\log(\rho(\mathcal{G}_{coin}))/\log(\rho(\mathcal{G}_{res})) =
1.47131$.}
\end{ex}

\begin{ex}
\label{Kenyon-Solomyak} {\em Continuing Ex. \ref{Dekking}, we also
looked at the self-affine tiling example in \cite[Fig.\,2 and
3]{KS} for which the eigenvalues of the expansion map satisfy
$x^3-x^2-4 x +3 = 0$ and the dominant eigenvalue is not a unit.
Let $\alpha\approx 2.19869$, $\beta\approx-1.91223$ be two real
roots of it. Set $Q=\begin{pmatrix}
\alpha&0\\0&\beta\end{pmatrix}$, ${\bf 1}=(1,1)$ and ${\bf
v}=(\alpha,\beta)$. Then the tile equation is
\begin{eqnarray*}
Q A_1&=& (A_3+{\bf v}) \cup (A_3+2{\bf v}-{\bf 1}) \cup (A_3+3{\bf v}-2\cdot{\bf 1})\\
Q A_2&=& (A_1+{\bf 1}) \cup (A_2+{\bf 1}) \cup
(A_2+{\bf v}) \cup (A_2+2{\bf v}-{\bf 1}) \cup (A_2+3{\bf v}-2\cdot{\bf 1})\\
Q A_3&=& A_1 \cup A_2
\end{eqnarray*}
It gives overlap coincidence as well. See Table\,\ref{Results}
with the notation ${\bf v}^2=(\alpha^2,\beta^2)$. }
\end{ex}

\begin{figure}
\begin{center}
\subfigure[Example \ref{Dekking}]
{
\psfig{figure=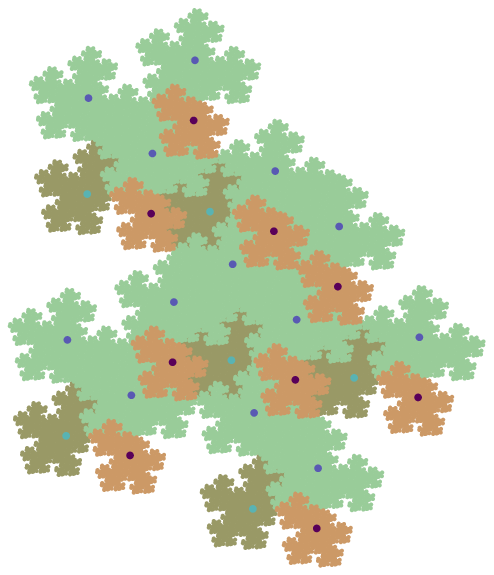}
\label{Ex7_5}
}
\hspace{2cm}
\subfigure[Example \ref{AFHI}]
{
\psfig{figure=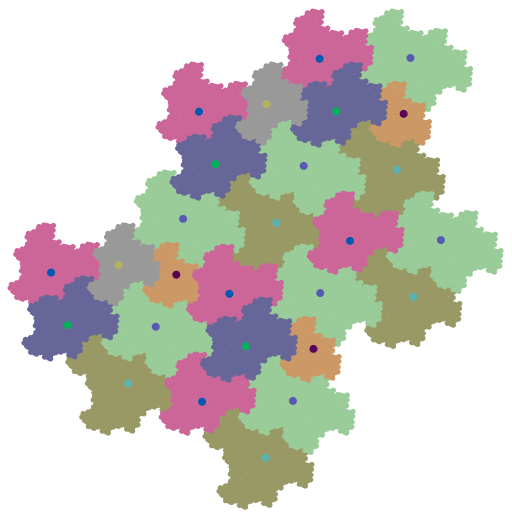}
\label{AFHI}
}
\end{center}
\end{figure}

\begin{ex} \label{Domino}
{\em It is known that domino tiling does not have pure point
spectrum \cite[Ex.\,7.3]{soltil}. The expansion map and MFS are
 \[Q = \begin{pmatrix} 2&0\\0&2\end{pmatrix} \ \ \ \mbox{and} \ \
 \Phi = \begin{pmatrix} \{f_3, f_4\} & \{f_1, f_6\} \\
\{f_1, f_5\} & \{f_2, f_4\} \end{pmatrix}\]
where $f_1=Qx$,
$f_2=Qx+(1,0)$, $f_3=Qx+(0,1)$, $f_4=Qx+(1,1)$, $f_5=Qx+(0,3)$ and
$f_6=Qx+(3,0)$. }
\end{ex}

\begin{ex} \label{minimal-Pisot}
{\em In the relation to the explicit construction of Markov
partition of toral automorphism, Thurston in \cite{Thu} introduced
$(d-1)$-dimensional self-similar tilings from numeration system
based on Pisot unit of degree $d$ which is called {\it Pisot dual
tilings}. Their basic properties are studied in Akiyama
\cite{Ak_Bd}. Such tiling dynamics are expected to be pure point.
We confirm that Pisot dual tilings associated to $x^3-x^2-x-1$,
$x^3-x-1$, $x^3-2x^2-x-1$, $x^3-3x^2-1$ and $x^4-x^3-x^2-x-1$
admit overlap coincidences from our algorithmic computation.

\noindent (1) The Pisot dual tiling associated to $x^3-x-1$ has an
expansive factor $\alpha \approx -0.877439 + i 0.744862 $ for
which $\alpha^3 + \alpha -1 =0$. The MFS $\Phi$ is {\tiny
\[
\begin{pmatrix}
\{f_1\} & \{f_1\}& \emptyset& \emptyset& \emptyset\\
\emptyset& \emptyset& \{f_1\}& \emptyset& \emptyset\\
\emptyset& \emptyset& \emptyset& \{f_1\}& \emptyset\\
\emptyset& \emptyset& \emptyset& \emptyset& \{f_1\}\\
\{f_2\}& \emptyset& \emptyset& \emptyset& \emptyset
\end{pmatrix} \]}
where $f_1 = \alpha x$ and $f_2 = \alpha x + 1$.

\noindent (2) The Pisot dual tiling associated to
$x^4-x^3-x^2-x-1$ is $3$-dimensional. The expansion map $Q$ and
MFS $\Phi$ are \[
\begin{pmatrix}
\frac{1}{\alpha_1} & 0 & 0 \\
0 & \frac{1}{\alpha_2} & 0 \\
0 & 0 & \frac{1}{\alpha_3} \end{pmatrix}
 \ \ \ \mbox{and} \ \
\begin{pmatrix} \{f_1\} & \{f_1\} &
\{f_1\} & \{f_1\} \\ \{f_2\} & \emptyset & \emptyset & \emptyset \\
\emptyset & \{f_2\} & \emptyset & \emptyset \\ \emptyset &
\emptyset & \{f_2\} & \emptyset \end{pmatrix},
\]
where $\alpha_1, \alpha_2, \alpha_3$ are roots of the polynomial
$x^4 - x^3 - x^2 - x -1$, $f_1 = Qx$ and $f_2 = Q(x + (1, 1, 1))$.
The translation vectors are $ (2 {\alpha_1}^2 - {\alpha_1}^3, 2
{\alpha_2}^2 - {\alpha_2}^3, 2 {\alpha_3}^2 - {\alpha_3}^3)$, $
(2{\alpha_1}+{\alpha_1}^2-{\alpha_1}^3,
2{\alpha_2}+{\alpha_2}^2-{\alpha_2}^3,
 2{\alpha_3}+{\alpha_3}^2-{\alpha_3}^3)$,
$(2+{\alpha_1}+{\alpha_1}^2-{\alpha_1}^3,
2+{\alpha_2}+{\alpha_2}^2-{\alpha_2}^3,
2+{\alpha_3}+{\alpha_3}^2-{\alpha_3}^3)$. }
\end{ex}

\begin{ex} \label{AFHI}
{\em Geometric realization of 1-dimensional substitutions has been
studied for a long time, which is motivated by Markov partition of
toral automorphism. The original idea came from Rauzy \cite{Rau}
and got extended in a great deal to Pisot substitutions in
\cite{AI} by Arnoux-Ito. They have a domain exchange structure
coming from substitutions and inherit their spectral properties
(see also \cite{FIR}). Recently Arnoux-Furukado-Harriss-Ito in
\cite{AFHI} generalized the idea to a special class of complex
Pisot substitution which no longer has direct domain exchange
structure but gives an explicit Markov partition of 4-toral
automorphism. We examined the example in \cite[Prop.\,6.8]{AFHI}
whose tile equation is
\begin{eqnarray*}
\alpha A_1&=&A_2\cup A_3\\
\alpha A_2&=&A_4 \cup A_5\\
\alpha A_3&=&A_6\\
\alpha A_4&=&A_1\\
\alpha A_5&=&A_2+\alpha-\alpha^2\\
\alpha A_6&=&A_4+1-\alpha^2
\end{eqnarray*}
with $\alpha \approx -0.727136 + i 0.934099 $ which is a root of
the polynomial $x^4-x^3+1$. The MFS $\Phi$ is {\tiny
\[  \begin{pmatrix}
\emptyset& \emptyset& \emptyset& \{f_1\} & \emptyset& \emptyset\\
 \{f_1\} & \emptyset& \emptyset& \emptyset & \{f_3\}& \emptyset\\
 \{f_1\} & \emptyset& \emptyset& \emptyset & \emptyset & \emptyset\\
\emptyset & \{f_1\} & \emptyset& \emptyset& \emptyset & \{f_2\}\\
\emptyset & \{f_1\} & \emptyset& \emptyset& \emptyset & \emptyset\\
\emptyset & \emptyset & \{f_1\} & \emptyset& \emptyset& \emptyset
\end{pmatrix}\]}
where $f_1 = \alpha x $, $f_2 = \alpha x + 1 - \alpha^2$ and $f_3
= \alpha x + \alpha-\alpha^2$. The result shows that
$\rho(\mathcal{G}_{coin}) \approx 1.40127$ and
$\rho(\mathcal{G}_{res}) \approx 1.22074$ and it implies overlap
coincidence. We note that the Hausdorff dimension of the boundary
of each tile is $1.18242$. \label{NonPisot}
}
\end{ex}

\begin{ex}
\label{Penrose} {\em Bandt-Gummelt in \cite{BG} gave two fractal
Penrose tilings by fractal kites and darts having exact matching
condition. We confirm that these tilings admit overlap
coincidence. The expanding matrix $Q=\begin{pmatrix} t &0
\\ 0& t
\end{pmatrix}$ where $t^2 - t - 1 = 0$ and the MFS $\Phi$ is $20 \times 20$ matrix such that $\Phi = (F_{ij})$ and
$F_{ij} = \{f: f=Qx+d, d \in \Dk_{ij}\}$. We give the digit set
matrix $(\Dk_{ij})$ in Figure\,\ref{Fractal-Penrose-tiling}, where
$w = \cos(\pi/5) + i \sin(\pi/5)$, $t = \frac{(1+ \sqrt{5})}{2}$,
and $c \in \C$ satisfies $f(t^2 i) = i$ with $f(z) =
\frac{z}{t}w^4 + c$.
Note that the Hausdorff dimension of the boundary of each tile is
1.26634.
\begin{figure}
\begin{center}
\includegraphics[width=0.8\textwidth]{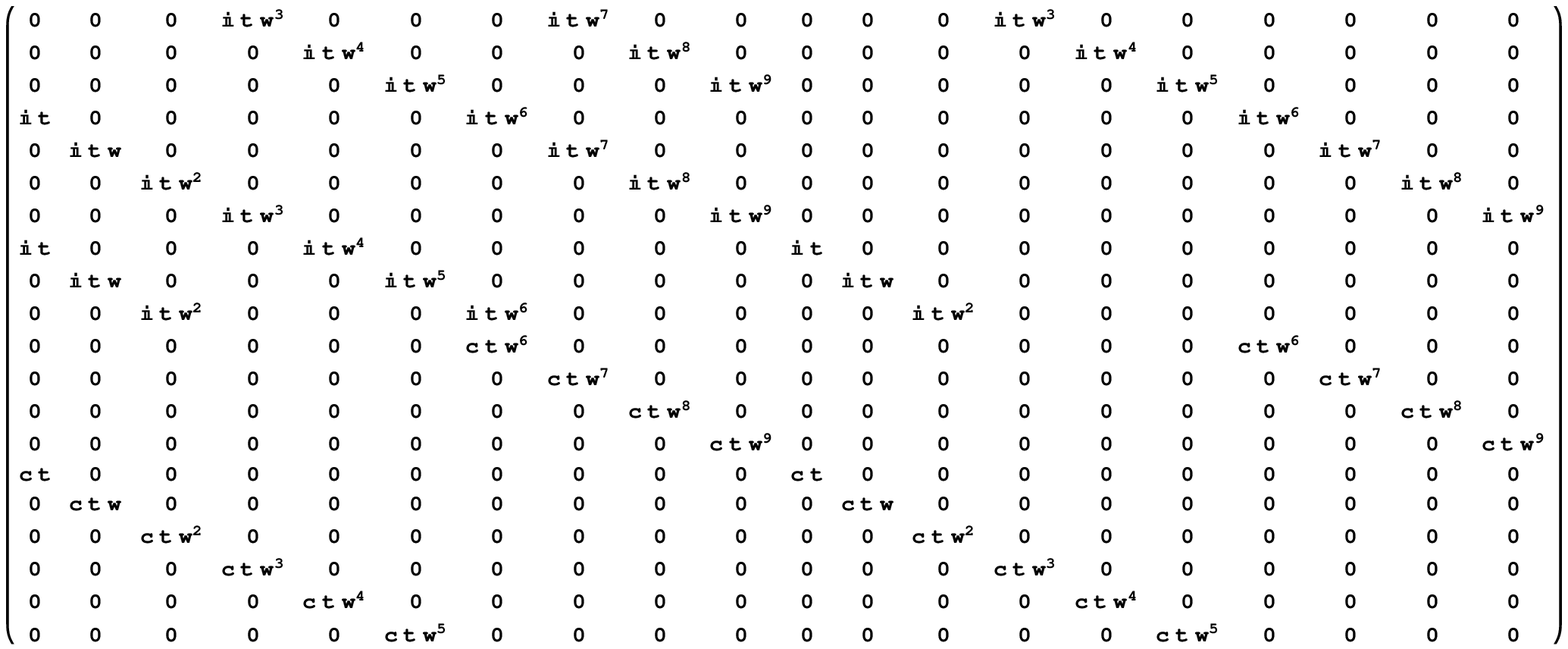}
\caption{Fractal Penrose tiling} \label{Fractal-Penrose-tiling}
\end{center}
\end{figure}
}
\end{ex}

\begin{ex} \label{3-dim-ChairTiling}
{\em Higher dimension chair tiling is discussed in \cite{LM1}. We
consider a 3-dim chair tiling which is defined by the expansion
matrix $\begin{pmatrix} 2&0&0 \\ 0&2&0 \\ 0&0&2 \end{pmatrix}$ and
the MFS $\Phi $ is {\tiny \[  \begin{pmatrix}
\{f_1, f_5\} & \{f_1\} & \{f_1\} & \{f_1\} & \emptyset & \{f_1\} & \{f_1\} & \{f_1\} \\
 \{f_2\} & \{f_2, f_6\} & \{f_2\} & \{f_2\} & \{f_2\} & \emptyset & \{f_2\} & \{f_2\} \\
 \{f_3\} & \{f_3\} & \{f_3, f_7\} & \{f_3\} & \{f_3\} & \{f_3\} & \emptyset & \{f_3\} \\
  \{f_4\} & \{f_4\} & \{f_4\} & \{f_4, f_8\} & \{f_4\} & \{f_4\} & \{f_4\} & \emptyset \\
 \emptyset & \{f_5\} & \{f_5\} & \{f_5\} & \{f_1, f_5\} & \{f_5\} & \{f_5\} & \{f_5\} \\
  \{f_6\} & \emptyset & \{f_6\} & \{f_6\} & \{f_6\} & \{f_2, f_6\} & \{f_6\} & \{f_6\} \\
   \{f_7\} & \{f_7\} & \emptyset & \{f_7\} & \{f_7\} & \{f_7\} & \{f_3, f_7\} & \{f_7\} \\
    \{f_8\} & \{f_8\} & \{f_8\} & \emptyset & \{f_8\} & \{f_8\} & \{f_8\} & \{f_4, f_8\}
 \end{pmatrix}\]}
where $f_1 = Qx + (0, 0, 0), f_2 = Qx + (1, 0, 0), f_3 = Qx + (0,
1, 0), f_4 = Qx + (1, 1, 0), f_5 = Qx + (1, 1, 1), f_6 = Qx + (0,
1, 1), f_7 = Qx + (1, 0, 1)$, and $f_8 = Qx + (0, 0, 1)$. }
\end{ex}

\begin{ex} \label{3-dim-ThueMorse}
{\em A $3$-dimension Thue-Morse tiling can be given by the
expanding matrix and the MFS \[ Q =
\begin{pmatrix} 2&0&0 \\ 0&2&0 \\ 0&0&2 \end{pmatrix} \ \ \ \mbox{and} \ \  \Phi = \begin{pmatrix}
\{f_1, f_4, f_6, f_7\} & \{f_2, f_3, f_8, f_5\} \\
  \{f_2, f_3, f_5, f_8\} & \{f_1, f_4, f_6, f_7\}
 \end{pmatrix} \]
where $f_i$, $1 \le i \le 8$, are given as above in
Ex.\,\ref{3-dim-ChairTiling}. }
\end{ex}

\begin{tiny}
\begin{table}
\begin{tabular}{|c|c|c|c|c|c|c|c|c|}
\hline
Tiling & Dim.  & Colour & \shortstack{\\ Translation \\ vectors} & $\# G_{coin}$ &
$\rho(G_{coin})$ & $\rho(G_{res})$ & \shortstack{\\ Pure \\ pointedness} \\
\hline
&&&&&&& \\
Ex.\,\ref{Fibonnaci} & 1
&  $1$ & $\alpha-1$ & 8 & 1.6180 & 1 & Yes \\
\hline
&&&&&&& \\
Ex.\,\ref{Dekking} & 2 &   2 & \shortstack{$-1-\alpha$\\$-\alpha-\alpha^2$}
& 15 & 1.4656 & 1.3247 & Yes \\
\hline
&&&&&&& \\
Ex.\,\ref{Kenyon-Solomyak} & 2 & 2& \shortstack{$-{\bf v}+{\bf v}^2$ \\ $-3-3{\bf v}+ {\bf v}^2$}
&10 & 4.2044 & 2.19869 & Yes \\
\hline
&&&&&&& \\
 Ex.\,\ref{Domino} & 2&  1 & \shortstack{$(1,0)$\\$(0,3)$} & 2 & 4 & 4 & No \\
\hline
&&&&&&& \\
 Ex.\,\ref{minimal-Pisot}(1) & 2 &  1&
 \shortstack{ $2-\alpha^2$ \\$-2-\alpha+2\alpha^2$ }
 & 20 & 1.3247 & 1.1673 & Yes \\
\hline
&&&&&&& \\
 Ex.\,\ref{AFHI} & 2 &  2 & \shortstack{$\alpha-\alpha^2$\\
$1-\alpha^2+\alpha^3$}
& 88 & 1.4013 & 1.2207 & Yes \\
\hline
&&&&&&& \\
 Ex.\,\ref{Penrose} & 2 & 1
&\shortstack{$(\sqrt{v},0)$\\$(\sqrt{v}/2,-v/2)$,\\$v^2-10v+5=0$}
&751 & 2.6180 & 1.8393 & Yes  \\
\hline
&&&&&&& \\
 Ex.\,\ref{3-dim-ChairTiling}& 3  &  1 &
 \shortstack{$(1,1,1)$ \\ $(0,0,2)$ \\ $(0,2,0)$}
& 16 & 8 & 4 & Yes   \\
\hline
&&&&&&& \\
Ex.\,\ref{3-dim-ThueMorse}  & 3 &   1 &
  \shortstack{ $(1,0,1)$\\ $(0,1,1)$\\ $(1,1,0)$} & 2 & 8 & 8 & No \\
\hline
&&&&&&& \\
 Ex.\,\ref{minimal-Pisot}(2) & 3 &  1 & See Ex.\,\ref{minimal-Pisot}(2) & 19 & 1.9276 & 1.6234 & Yes \\
\hline
\end{tabular}
\vspace{2mm} \caption{} \label{Results}
\end{table}
\end{tiny}

\medskip

\noindent {\bf Acknowledgments.}   We thank Boris Solomyak and
Beno\^it Loridant for helpful discussions.


\begin{thebibliography}{99}
\bibitem{Ak_Bd}  S. Akiyama,  On the boundary of self affine tilings generated
by Pisot numbers, {\em J. Math. Soc. Japan},  {\bf 54}(2) (2002),
283-308.

\bibitem{AI} P. Arnoux and S. Ito, Pisot substitutions and Rauzy
fractals. {\em Bull.\ Belg.\ Math.\ Soc.} {\bf 8} (2001) 181--207.

\bibitem{AFHI}
P. Arnoux, M. Furukado, E. Harriss, and Sh. Ito, Algebraic
numbers, free group automorphisms and substitutions of the plane,
(2010) to appear in {\em Trans. Amer. Math. Soc.}

\bibitem{BL}
M. Baake and D. Lenz, Dynamical systems on translation bounded
measures : pure point dynamical and diffraction spectra, {\em
Ergod. Th.\ \& Dynam.\ Sys.} {\bf 24} (2004), 1867--1893.

\bibitem{BG}
C. Bandt, P. Gummelt, Fractal Penrose tilings. I. Construction and
matching rules, {\em Aequationes Math.} {\bf 53}(3) (1997),
295--307.

\bibitem{BK} M. Barge and J. Kwapisz, Geometric theory of
unimodular Pisot substitution. {\em Amer.\ J.\ Math.} {\bf 128}
(2006), 1219--1282.

\bibitem{BBK} V. Baker, M. Barge and J. Kwapisz, Geometric
realization and coincidence for reducible non-unimodular Pisot
tiling spaces with an application to $\beta$-shifts, {\em Ann.
Inst.Fourier(Grenoble)} {\bf 56}(7) (2006), 2213-2248.

\bibitem{Clark-Sadun}
A. Clark and L. Sadun, When size matters: subshifts and their
related tiling spaces, {\em Ergodic Theory Dynam. Systems} {\bf
23}(4) (2003), 1043--1057.

\bibitem{Deka} M. Dekking, Replicating superfigures and endomorphisms of free groups, {\em J.
Combin. Theory Ser. A} {\bf 32}(3) (1982), 315--320.

\bibitem{Dekb} M. Dekking, Recurrent sets, {\em Adv. in Math.} {\bf 44}(1) (1982), 78--104.


\bibitem{Dek} M. Dekking, The spectrum of dynamical systems arising from substitutions of constant
length, {\em Z.\ Wahrsch.\ Verw.\ Gebiete} {\bf 41} (1978),
221--239.

\bibitem{DS} D. Frettloh and B. Sing, Computing modular coincidences
for substitution tilings and point sets. {\em Discrete Comput.
Geom.} {\bf 37}(3) (2007), 381--407.


\bibitem{FIR} M.~Furukado, Sh.~Ito and H.~Rao, Geometric realizations
of hyperbolic unimodular substitutions, {\em Progress in
Probability} {\bf 61} (2009), 251--268.


\bibitem{Gouere} J.-B. Gou\'{e}r\'{e}, Diffraction and Palm measure of point processes,
{\em C. R. Acad. Sci. Paris} {\bf 336}(1) (2003), 57--62.

\bibitem{He-Lau:08}
X.-G. He and K.-S. Lau, On a generalized dimension of self-affine
  fractals, {\em Math. Nachr.} {\bf 281}(8) (2008), 1142--1158.


\bibitem{IR} S. Ito and H. Rao, Atomic surfaces, tiling and
coincidence I. Irreducible case. {\em Israel J. Math.} {\bf 153}
(2006), 129--156.

\bibitem{Ken} R. Kenyon, The construction of self-similar tilings,
{\em Geom. Funct. Anal.} {\bf 6}(3) (1996), 471--488.

\bibitem{KS} R. Kenyon and B. Solomyak, On the characterization of expansion maps
for self-affine tilings, (2010) to appear in {\em Discrete Compu.
Geom.}

\bibitem{Lag} J. C. Lagarias, Meyer's concept of quasicrystal and quasiregular sets,
{\em Comm. Math. Phys.} {\bf 179} (1996), 365-376.

\bibitem{Lag99} J. C. Lagarias, Geometric models for quasicrystals
I. Delone sets of finite type, {\em Discrete Compu. Geom.} {\bf
21} (1999), 161--191.

\bibitem{LP} J. C. Lagarias, P. A. B. Pleasants, Repetitive Delone
sets and quasicrystals, {\em Ergodic Th. Dynam. Sys.} {\bf 23}
(2003), 831--867.

\bibitem{lawa} J. C. Lagarias and Y. Wang,
Substitution Delone sets, {\em Discrete Comput. Geom.} {\bf 29} (2003),
175--209.

\bibitem{Lee} J.-Y. Lee, Substitution Delone sets with pure point spectrum are inter-model sets,
{\em Journal of Geometry and Physics} {\bf 57} (2007), 2263-2285.

\bibitem{LM1} J.-Y. Lee and R. V. Moody,
Lattice substitution systems and model sets, {\em Discrete Comput.
Geom.} {\bf 25} (2001), 173--201.

\bibitem{LMS1} J.-Y. Lee, R. V. Moody, and B. Solomyak,
Pure point dynamical and diffraction spectra, {\em Ann. Henri
Poincar{\'e}} {\bf 3} (2002), 1003--1018.

\bibitem{LMS2} J.-Y. Lee, R. V. Moody, and B. Solomyak,
Consequences of pure point diffraction spectra for multiset
substitution systems, {\em Discrete Comp. Geom.} {\bf 29} (2003),
525--560.

\bibitem{LS} J.-Y. Lee and B. Solomyak, Pure point diffractive substitution
Delone sets have the Meyer property, {\em Discrete Comp. Geom.}
{\bf 39} (2007), 319-338.

\bibitem{LS2} J.-Y. Lee and B. Solomyak, Pisot family substitution
tilings, discrete spectrum and the Meyer property, submitted.

\bibitem{Luo-Yang:09}
J.~Luo and Y.-M. Yang, On single-matrix graph-directed iterated
function systems, preprint.

\bibitem{Mauldin-Williams:88}
R.~D. Mauldin and S.~C. Williams, Hausdorff dimension in graph
directed constructions, {\em Trans. Amer. Math. Soc.} {\bf 309}
(1988), 811--829.

\bibitem{Meyer} Y. Meyer, Algebraic numbers and harmonic analysis,
{\em North Holland}, 1970.

\bibitem{Praggastis:99}
B.~Praggastis, Numeration systems and Markov partitions from
  self-similar tilings, {\em Trans. Amer. Math. Soc.} {\bf 351}(8) (1999), 3315--3349.

\bibitem{RVM97} R. V. Moody, Meyer sets and their duals, in {\em The Mathematics of Long-Range Aperiodic
Order (Waterloo, ON, 1995)}, R.~V.~Moody, ed., NATO Adv.\ Sci.\ Inst.\ Ser.\ C Math.\
 Phys.\ Sci., Vol. 489, Kluwer Acad.\ Publ., Dordrecht, 1997, 403--441.

\bibitem{Fogg} N. Pytheas Fogg, Substitutions in dynamics,
arithmetics and combinatorics, {\em Lecture notes in mathematics},
V. Berth\'{e} and S. Ferenczi and C. Mauduit and A. Siegel, ed.
Springer-Verlag, 2002.


\bibitem{Rau} G.~Rauzy,  Nombres Alg\'{e}briques et substitutions,  {\em Bull. Soc.
France} {\bf 110} (1982), 147--178.

\bibitem{Schief:94} A. Schief, Separation properties for self-similar sets,
{\em Proc. Amer. Math. Soc.} {\bf 122}(1) (1994), 111--115.

\bibitem{Siegel} A. Siegel, Pure discrete spectrum dynamical
system and periodic tiling associated with a substitution. {\em
Ann. Inst. Fourier (Grenoble)} {\bf 54}(2) (2004), 341--381.

\bibitem{ST} A. Siegel and J. Thuswaldner, Topological properties
of Rauzy fractals. {\em Memoire de la SMF}. To appear.

\bibitem{Sing} B. Sing, Pisot substitutions and beyond. Ph.D.
thesis, {\em Universitat Bielefeld}, (2006)

\bibitem{SS} V. F. Sirvent and B. Solomyak, Pure discrete spectrum for one-dimensional
substitution systems of Pisot type. {\em Canad. Math. Bull.} {\bf
45}(4) (2002) 697--710. Dedicated to Robert V. Moody.

\bibitem{soltil} B. Solomyak, Dynamics of self-similar tilings,
{\em Ergodic Th.\ Dynam.\ Sys.} {\bf 17} (1997), 695--738.
Corrections to
`Dynamics of self-similar tilings',
{\em ibid.} {\bf 19} (1999), 1685.

\bibitem{sol3} B. Solomyak, Eigenfunctions for substitution tiling systems, {\em Advanced Studies in Pure Mathematics} {\bf 43},
(2006), International Conference in Probability and Number Theory,
Kanazawa, (2005) 1--22.

\bibitem{Solomyak:08}
B.~Solomyak, Tilings and dynamics, preprint, Lecture Notes for EMS
Summer School on Combinatorics, Automata and Number Theory, 2006,
Liege.

\bibitem{Thu} W. P. Thurston, Groups, Tilings and finite state automata,
{\em AMS Colloquium lectures}, 1989.

\end{thebibliography}
\end{document}